\documentclass{amsart}

\usepackage{amsfonts, amsmath, amssymb, amsthm}
\usepackage{mathrsfs}
\usepackage{array, multirow}
\usepackage{tikz, graphicx}
\usepackage{textcomp}
\usetikzlibrary{positioning, scopes}

\usepackage[foot]{amsaddr}

\usepackage[margin=1.3in]{geometry}

\newcommand{\real}{\mathbb{R}}

\newtheorem{theorem}{Theorem}[section]
\newtheorem{lemma}[theorem]{Lemma}
\newtheorem{corollary}[theorem]{Corollary}
\newtheorem{question}[theorem]{Question}

\newtheorem{proposition}[theorem]{Proposition}
\newtheorem{definition}[theorem]{Definition}

\newtheorem*{proposition*}{Proposition}

\theoremstyle{definition} 
\newtheorem{remark}[theorem]{Remark}

\begin{document}

\title[AREA AND ANTIPODAL DISTANCE IN CONVEX HYPERSURFACES]{Area and antipodal distance in convex hypersurfaces}

\author[J. DIBBLE]{JAMES DIBBLE$^a$}
\address{$^a$Department of Mathematics and Statistics, University of Southern Maine, Portland, ME 04101}
\email{james.dibble@maine.edu}

\author[J.A. HOISINGTON]{JOSEPH ANSEL HOISINGTON$^b$}
\address{$^b$Department of Mathematics, Rose--Hulman Institute of Technology, Terre Haute, IN 47803}
\email{hoisingt@rose-hulman.edu}


\date{}

\begin{abstract}
We establish a lower bound for the surface area of a closed, convex hypersurface in Euclidean space in terms of its displacement under continuous maps. As a result, a hypothesized lower bound for the volume of a Riemannian $n$-sphere, proved by Berger in dimension $n=2$ and disproved by Croke in dimensions $n \geq 3$, is valid for convex hypersurfaces in all dimensions. We also establish a sharp lower bound for the mean width of a convex hypersurface.  
\end{abstract}

\maketitle


\section{Introduction}
\label{intro} 


The primary goal of this paper is to establish the following lower bound for the $n$-dimensional Hausdorff measure of a closed, convex hypersurface in terms of its metric properties.

\begin{theorem}\label{main theorem}
For each natural number $n$, there exists $h_n > 0$ such that if $M^n$ is a closed, convex hypersurface in $\real^{n+1}$ with intrinsic distance $d_{M}$ and $n$-dimensional Hausdorff measure $\mathrm{Area}(M)$, $\alpha$ is a continuous map from $M$ to itself, and $\mu(\alpha) = \displaystyle \min_{x \in M} d_M(x,\alpha(x))$, then
\[
\mathrm{Area}(M) > h_n \mu(\alpha)^n.
\]
\end{theorem}
\noindent Below, we discuss a question originating in the work of Berger \cite{Berger1977} that Theorem \ref{main theorem} partially addresses. 
\smallskip 

Throughout this paper, we refer to the $n$-dimensional Hausdorff measure of a convex hypersurface in $\real^{n+1}$ as its area, to distinguish it from the $(n+1)$-dimensional volume of the domain it encloses. A fundamental rigidity result for convex hypersurfaces, proved in increasing generality in work of Cauchy \cite{AignerZiegler1999,Cauchy1813}, Cohn-Vossen \cite{CohnVossen1936}, Herglotz \cite{Herglotz1943}, Pogorelov \cite{Pogorelov1973}, Sen'kin \cite{Senkin1972}, and, recently in full generality, Borisenko \cite{Borisenko2025}, states that, for $n \geq 2$, closed, convex hypersurfaces in $\real^{n+1}$ that are intrinsically isometric are congruent by an isometry of $\real^{n+1}$. In this sense, the intrinsic geometry of a convex hypersurface determines its extrinsic geometry. The proof of Theorem \ref{main theorem} is based on the interplay between these intrinsic and extrinsic geometries. One of the estimates in the proof is the following lower bound for the area of a convex hypersurface in terms of the ratio between its intrinsic and extrinsic distances. 

\begin{proposition}\label{intrinsic-extrinsic proposition 1}
For each natural number $n$, there exists an increasing, continuous function $\mathcal{I}_n : (1,\infty) \to (0, \frac{\omega_n}{2^{n-1}})$, where $\omega_n$ is the volume of the unit $n$-ball, such that if $M^n$ is a closed, convex hypersurface in $\real^{n+1}$ with intrinsic metric $d_{M}$ and $x,y$ are points in $M$ with $\frac{d_M(x,y)}{|y-x|} \geq \rho$, then
\[
\mathrm{Area}(M) > \mathcal{I}_n(\rho) d_M(x,y)^n.
\]
\end{proposition}
\noindent To the best of our knowledge, this result, which may be of independent interest, is not found elsewhere in the literature. The proof gives an explicit formula for $\mathcal{I}_n(\rho)$, recorded in Remark \ref{formula for I}. This formula is unlikely to be optimal, but we show that it is suboptimal by at most a factor of $(n-1)\pi$ as $\rho \to \infty$. In some important special cases, the proof gives a result that is asymptotically optimal, recorded in Corollary \ref{asymptotically optimal cases}. In Proposition \ref{enclosed volume bound}, we establish a lower bound for the volume of the domain enclosed by a convex hypersurface $M^n$ in $\real^{n+1}$ that is, in a sense, complementary to Proposition \ref{intrinsic-extrinsic proposition 1}. 
\smallskip 

Understanding the relationship between the metric properties of a space and its volume is an important problem, with a rich history, in many branches of geometry. In Riemannian geometry, where an especially important goal is to find volume bounds that don't depend explicitly on curvature bounds, two of the most fundamental and well-known results of this type are Berger's isoembolic inequality \cite{Berger1980,Besse1978} and Gromov's filling radius estimate \cite{Gromov1983}. In his early work on the isoembolic inequality, Berger proved that the area of a Riemannian metric on the $2$-sphere can be bounded below in terms of its displacement under fixed-point-free involutions, in the following sense.

\begin{theorem}[\cite{Berger1977}, Proposition 2]\label{berger antipodal bound}
Let $g$ be a Riemannian metric on the $2$-sphere with distance function $d_g$ and area $\mathrm{Area}(S^2,g)$. If $\alpha$ is any fixed-point-free diffeomorphism of the $2$-sphere of order $2$ and $\mu(\alpha) = \min\limits_{x \in S^2} d_{g}(x,\alpha(x))$, then
\begin{align}\label{berger_antipodal_inequality}
\displaystyle \mathrm{Area}(S^2,g) \geq \widehat{h}_2 \mu(\alpha)^2,
\end{align}where $\widehat{h}_2$ is a positive constant that is at least $\frac{1}{2}$. 
\end{theorem}
\noindent Berger then used Theorem \ref{berger antipodal bound} to establish a curvature-free lower bound for the volume of a Riemannian $3$-manifold in terms of its injectivity radius in \cite[Proposition 1]{Berger1977}.  Berger later proved the sharp isoembolic inequality in all dimensions by a different argument in \cite{Berger1980}, but strengthening and generalizing Theorem \ref{berger antipodal bound} became a goal in its own right, cf. \cite{CrokeKatz2003}. It was not known whether Riemannian metrics on higher-dimensional spheres satisfy a lower bound for volume in terms of displacement as in Theorem \ref{berger antipodal bound} until Croke showed that this is not the case: for each $n \geq 3$, the results in \cite{Croke2008} give a family of Riemannian metrics $g$ on $S^{n}$ and fixed-point-free involutions $\alpha$ of $S^{n}$ for which the infimum of the ratio $\frac{\mathrm{Vol}(S^{n},g)}{\mu(\alpha)^n}$ is zero. 
\smallskip 

Theorem \ref{main theorem} establishes a generalization of Berger's inequality for convex hypersurfaces in all dimensions. Theorem \ref{main theorem} is broader than Theorem \ref{berger antipodal bound} in that it gives a lower bound in terms of displacement under any continuous map. Croke gave a proof of Theorem \ref{berger antipodal bound} in \cite[Theorem 1.4]{Croke2002} that establishes the inequality \eqref{berger_antipodal_inequality} for any continuous self-map of the $2$-sphere, although the result is stated only for involutions. However, the constant $\widehat{h}_2 = \frac{1}{3{,}456}$ in Croke's result is smaller than in Theorem \ref{berger antipodal bound}. With presently known constants, including the values for the constants in Theorem \ref{main theorem} given below, none of Theorem \ref{berger antipodal bound}, the result proved in \cite[Theorem 1.4]{Croke2002}, or the case $n=2$ of Theorem \ref{main theorem} follows from the others. We discuss the relationships among these results, and related aspects of the topological classification of fixed-point-free involutions of spheres, in Remark \ref{Yang discussion}. 
\smallskip 

Several statements that would strengthen and extend Theorem \ref{main theorem} are precluded by known results.  For example, Theorem \ref{main theorem} is not valid for star-shaped hypersurfaces, because the examples constructed in \cite{Croke2008} are boundaries of star-shaped domains in $\real^{n+1}$. However, there is a possible generalization of Theorem \ref{main theorem} that seems to be a natural goal for future work: convex hypersurfaces have nonnegative curvature as length spaces, and, conversely, work of Sacksteder \cite{Sacksteder1960} implies that closed hypersurfaces in $\real^{n+1}$,  $n \geq 2$, with nonnegative sectional curvature are embedded as convex hypersurfaces. It therefore seems possible that Theorem \ref{main theorem} can be subsumed into a result about spaces with nonnegative curvature. More broadly, it seems natural to ask the following.

\begin{question}\label{nonpositive curvature question}
Is it the case that, for any $K \in \real$ and $\delta > 0$, there exists $\widetilde{h}_n(K,\delta) > 0$ such that if $g$ is a Riemannian metric on $S^n$ with sectional curvature at least $K$ and diameter at most $\delta$, $\alpha$ is any map (or involution) from $S^n$ to itself, and $\mu(\alpha) = \min\limits_{x \in S^n} d_{g}(x,\alpha(x))$, then $\mathrm{Vol}(S^n,g) \geq \widetilde{h}_n(K,\delta) \mu(\alpha)^n$? \end{question}
\noindent An affirmative answer to Question \ref{nonpositive curvature question} would be consistent with all results known to the authors. The constant $\widetilde{h}_n(K,\delta)$ in a result of this type would be independent of $\delta$ for $K \geq 0$, so such a result would give a uniform constant for metrics with nonnegative curvature and would therefore supersede Theorem \ref{main theorem} on a qualitative level. We discuss another question related to possible extensions of Theorem \ref{main theorem} in Remark \ref{final remark}. 
\smallskip 

The proof of Theorem \ref{main theorem} gives explicit but nonoptimal values for the constants $h_n$. For $n = 2$, one can take $h_2 = \frac{\sqrt[3]{\frac{\pi}{6}}}{(1 + \sqrt[6]{\frac{\pi}{6}})^2} \approx 0.2237$. The values established for $h_3 \approx 0.0443$ and $h_4 \approx 0.0080$ can also be expressed in radicals but are significantly more complicated. In general, the value for $h_n$ given by our proof is a root of a polynomial of degree $n$ and does not seem to have an expression in radicals in terms of familiar geometric constants for $n \geq 5$. Relative to the constants given by the antipodal map of the round sphere, these constants decay at a rate roughly proportional to $\frac{1}{\sqrt{n!}}$, as shown in Proposition \ref{optimality_estimate}. A stronger result in Proposition \ref{intrinsic-extrinsic proposition 1} would give an improved estimate for the constants in Theorem \ref{main theorem}, but, even with an optimal result in Proposition \ref{intrinsic-extrinsic proposition 1}, our proof of Theorem \ref{main theorem} would not immediately give the optimal result. At several points in the paper, we discuss ways that our estimate for the constants in Theorem \ref{main theorem} could be improved in conjunction with other conjectures and results. In Appendix \ref{asymptotics_appendix}, we discuss the asymptotic behavior of our estimate for the constants in Theorem \ref{main theorem} in more detail. 
\smallskip 

The optimal constant in Theorem \ref{berger antipodal bound} is also not known but is expected to be $\frac{4}{\pi}$, the constant given by the antipodal map of the round sphere.  We will establish the following sharp lower bound for the mean width of a convex hypersurface. If the optimal constant in either Theorem \ref{main theorem} or \ref{berger antipodal bound} were given by the antipodal map of the round sphere, the case $n=2$ of this result would be an immediate corollary by Minkowski's inequality \cite{Minkowski1903,Schneider1993}. 
\begin{theorem}\label{mean width}
Let $M^n$ be a closed, convex hypersurface in $\real^{n+1}$, $\alpha$ a continuous map from $M$ to itself, $\mu(\alpha)$ its minimum intrinsic displacement, as in Theorem \ref{main theorem}, and $\Xi(M)$ the mean width of $M$, as in Definition \ref{width definitions}. Then,
\begin{equation}
\label{mean width eqn}
\displaystyle \Xi(M) \geq \frac{2}{\pi} \mu(\alpha). 
\end{equation}
For $n \geq 2$, equality holds if and only if $M$ is a round sphere and $\alpha$ is the antipodal map.
\end{theorem}
\noindent We give two proofs of Theorem \ref{mean width} and discuss the relationship between these two arguments in Section \ref{main results}. Theorem \ref{mean width} implies a sharp lower bound for the mean curvature of smooth convex hypersurfaces, recorded in Corollary \ref{mean curvature}.  For the $n = 1$ case of Theorem \ref{mean width}, see Remark \ref{crofton remark}. 
\smallskip 

We conclude this introduction by briefly reviewing some results related to Theorems \ref{main theorem}, \ref{berger antipodal bound}, and \ref{mean width}, giving an outline of the rest of the paper, and setting some notation. Although Theorem \ref{berger antipodal bound} does not extend to dimensions $n \geq 3$, Croke proved in \cite{Croke1987} that, in a precise sense, the antipodal map of the round Riemannian metric on $S^n$ is infinitesimally optimal for the ratio $\frac{\mathrm{Vol}(S^n,g)}{d(x,\alpha(x))^n}$ in all dimensions among smooth involutions $\alpha$ that are conjugate to the antipodal map by diffeomorphisms.  In \cite{Croke2002}, Croke proved that, for any Riemannian metric $g$ on $S^3$,
\[
\mathrm{Vol}(S^3,g) \geq \widehat{h}_3 \big( \min\{ \Lambda(g), \mu(\alpha) \} \big)^3,
\]
where $\widehat{h}_3 > 0$ is a universal constant, $\Lambda(g)$ is the length of the shortest closed geodesic in $(S^3,g)$, and $\mu(\alpha)$ is the minimum displacement of an arbitrary fixed-point-free involution $\alpha$ of $S^3$. This result also partially addresses another long-standing question in Riemannian geometry: whether the volume of a Riemannian metric on the $n$-sphere, or any closed manifold, can be bounded below in terms of the length of its shortest closed geodesic. This is known to be true for the $2$-sphere by work of Croke \cite{Croke1988}, open in general for the sphere of dimension $n \geq 3$, and true for convex hypersurfaces in all dimensions by work of Treibergs \cite{Treibergs1985} and Croke \cite{Croke1988}; however, the sharp inequality between area and the length of the shortest closed geodesic is not known even for convex surfaces in $\real^3$. In connection with Theorem \ref{mean width}, \'{A}lvarez Paiva proved a sharp lower bound for the mean width of a convex surface in $\real^{3}$ in terms of the length of its shortest closed geodesic in \cite{Paiva97}.  The authors have recently established a sharp extension of \'{A}lvarez Paiva's result in all dimensions, together with a characterization of the equality case, in \cite{DibbleHoisington2026b}.  
\smallskip 

\noindent \textbf{Organization of the paper.} Section \ref{int greater than ext} contains the proof of Proposition \ref{intrinsic-extrinsic proposition 1}. Section \ref{int comp to ext} contains the proofs of Proposition \ref{enclosed volume bound} and Corollary \ref{intrinsic-extrinsic proposition 2}. Section \ref{main results} contains the proofs of Theorems \ref{main theorem} and \ref{mean width}. Theorem \ref{main theorem} is proved as a corollary of Proposition \ref{area bound proposition}, which gives a more detailed description of the relationship between the metric properties of a convex hypersurface and its area. Appendix \ref{asymptotics_appendix} describes the asymptotics of our estimate for the constants in Theorem \ref{main theorem}, and Appendix \ref{int_ext_appendix} contains a family of results about an invariant for fixed-point-free maps appearing in the proof of Theorem \ref{main theorem}. The results in Sections \ref{intro}, \ref{int greater than ext}, \ref{int comp to ext}, and \ref{main results} do not depend on the appendices.
\smallskip 

\noindent \textbf{Notation.} We denote by $\mathrm{Vol}(\Omega)$ the $(n+1)$-dimensional Hausdorff measure of a domain $\Omega^{n+1}$ in $\real^{n+1}$ and by $\mathrm{Area}(M)$ the $n$-dimensional Hausdorff measure of a hypersurface $M^n$ in $\real^{n+1}$. For $x,y \in M$, $d_M(x,y)$ denotes the intrinsic distance between $x$ and $y$, i.e., the minimal length of a path in $M$ connecting them. All maps are assumed to be continuous. If $\alpha$ is any map from a length space $X$ to itself, then $\mu(\alpha) = \displaystyle \min_{x \in X} d_X (x,\alpha(x))$ is the minimum intrinsic displacement  of $\alpha$. The usual norm on $\real^{n+1}$ is denoted by $|\cdot|$, so that $|y - x|$, the distance between $x$ and $y$, is the length of the line segment $\overline{xy}$. The line in $\real^{n+1}$ connecting $x$ and $y$ is denoted by $\ell_{xy}$, and, if an affine subspace containing $\ell_{xy}$ is understood, the orthogonal complement to $\ell_{xy}$ at $x$ is denoted by $\ell_{xy}^\perp$. The length of a curve $\kappa$ is denoted by $L(\kappa)$. Orthogonal projection onto an affine subspace $X$ of $\real^{n+1}$ is denoted by $\phi_X$. The volume of the $n$-dimensional unit ball in $\real^n$ and the area of the $n$-dimensional unit sphere in $\real^{n+1}$ are denoted by $\omega_n$ and $\sigma_n$, respectively. If $M^n$ is a convex hypersurface in $\real^{n+1}$ and $\alpha$ is a map from $M$ into itself, then $\rho(\alpha) = {\displaystyle \max_{x \in M}} \frac{d_M(x,\alpha(x))}{|\alpha(x) - x|}$ is the maximum ratio of the intrinsic to extrinsic displacements of $\alpha$. The width of $M$ along a line $\ell$ is $w_\ell(M)$, the minimum such width over the Grassmannian $Gr(1,n+1)$ is $\underline{w}(M)$, and the mean width of $M$ is $\Xi(M)$.
\smallskip

\noindent \textbf{Acknowledgements.} We are happy to thank Chris Croke and Rob Kusner for helpful conversations about this work and Matteo Raffaelli for suggesting improvements to the manuscript. Figures were created using GeoGebra.


\section{When intrinsic distance is large relative to extrinsic distance}
\label{int greater than ext}


The aim of this section is to prove Proposition \ref{intrinsic-extrinsic proposition 1}. The basic idea of the proof is as follows: if the intrinsic distance between points $x,y$ in a closed, convex hypersurface $M^n$ in $\real^{n+1}$ is much greater than their extrinsic distance, any $2$-plane containing $x$ and $y$ must intersect $M$ in a convex curve that is long relative to $|y - x|$. The existence of this family of large plane sections implies there is a hyperplane onto which $M$ has a large orthogonal projection, which gives a lower bound for the area of $M$. It may be helpful to think of a surface resembling a flat pancake, with points whose intrinsic distance is large relative to their extrinsic distance, and which has large area, although it may enclose an arbitrarily small volume. 

\begin{definition}\label{width definitions}
Let $M^n$ be a closed, convex hypersurface in $\real^{n+1}$ and $\ell$ a line in $\real^{n+1}$. The \textbf{width $w_\ell(M)$ of $M$ along $\ell$} is the length of the orthogonal projection of $M$ onto $\ell$ or, equivalently, the distance between the two distinct supporting hyperplanes of $M$ that are orthogonal to $\ell$. The \textbf{minimum width $\underline{w}(M)$ of $M$} is the minimum of $w_\ell(M)$ as $\ell$ varies over $Gr(1,n+1)$, the Grassmannian of unoriented lines through the origin in $\real^{n+1}$. The \textbf{mean width} of $M$ is
\[
\Xi(M) = \int_{Gr(1,n+1)} w_\ell (M) \, d\ell,
\]
where $d \ell$ is the $SO(n+1)$-invariant measure on $Gr(1,n+1)$ with total measure $1$.
\end{definition}

\begin{theorem} [Crofton's formula \cite{Crofton1868,Santalo2004}]\label{crofton's formula}
Let $\kappa$ be a closed, convex plane curve and $L(\kappa)$ its length. Then, $L(\kappa) = \pi \, \Xi(\kappa)$.
\end{theorem}
\begin{remark}\label{crofton remark}
By Theorem \ref{crofton's formula}, Theorem \ref{mean width} gives a sharp version of Theorem \ref{main theorem} for convex curves. The equality case for curves in Theorem \ref{mean width} is larger than for convex hypersurfaces of dimension $n \geq 2$ because all convex curves admit a unique mapping $\alpha$, sending each point $x$ to the unique point $\alpha(x)$ at maximum intrinsic distance from $x$, for which equality holds in (\ref{mean width eqn}). 
\end{remark}
\noindent A more general version of Crofton's formula, discussed in \cite{Santalo2004,Schneider1993} and many other texts on convex and integral geometry, has the following implication.

\begin{proposition}\label{crofton corollary}
If $M_1$ and $M_2$ are closed, convex hypersurfaces in $\real^{n+1}$ and $M_1$ lies in the domain enclosed by $M_2$, then $\mathrm{Area}(M_1) \leq \mathrm{Area}(M_2)$, with equality if and only if $M_1 = M_2$.	
\end{proposition}

\begin{lemma}\label{max width bound}
Let $\kappa$ be a closed, convex plane curve. Then, for all distinct $x,y \in \kappa$,
\[
\max \{ w_{\ell_{xy}} (\kappa), w_{\ell_{xy}^\perp} (\kappa) \} \geq \frac{1}{2} d_\kappa (x,y).
\]
\end{lemma}

\begin{proof} Because $\kappa$ is contained in the compact region bounded by the four support lines that define $w_{\ell_{xy}} (\kappa)$ and $w_{\ell_{xy}^\perp} (\kappa)$, Proposition \ref{crofton corollary} implies that
\[
2d_\kappa (x,y) \leq L(\kappa) \leq 2[w_{\ell_{xy}} (\kappa) + w_{\ell_{xy}^\perp} (\kappa)] \leq 4 \max \{ w_{\ell_{xy}} (\kappa), w_{\ell_{xy}^\perp} (\kappa) \}.
\]
\end{proof}

\begin{lemma}\label{width lower bound for curves}
Let $\kappa$ be a closed, convex plane curve. Suppose distinct $x,y \in \kappa$ satisfy $d_\kappa (x,y) \geq \rho|y - x|$ for some $\rho > 1$. If the line $\ell_{xy}^\perp$ is a support line for $\kappa$, then
\[
w_{\ell_{xy}^\perp} (\kappa) \geq \Big( \frac{\rho - 1}{\rho} \Big) d_\kappa (x,y).
\]
\end{lemma}

\begin{proof} Since $d_\kappa (x,y) > |y - x|$, $\ell_{xy}$ meets $\kappa$ only at $x$ and $y$, and therefore the two support lines for $\kappa$ that are parallel to $\ell_{xy}$ must lie on opposite sides of $\ell_{xy}$. In particular, $\ell_{xy}$ is not a support line for $\kappa$. Thus, $\kappa \setminus \{ x,y \}$ consists of two open arcs, one in each half-space bounded by $\ell_{xy}$. 
\smallskip 

If $\ell_y^s$ is any support line to $\kappa$ at $y$, then at least one of the two rays of $\ell_y^s$ beginning at $y$ must form an acute angle with the segment $\overline{xy}$ between $x$ and $y$. Denote by $\kappa_a$ the convex curve consisting of the segment $\overline{xy}$ and the component of $\kappa$ on the side containing such an acute angle, and let $\theta$ be the acute angle between $\ell_y^s$ and $\ell_{sy}^\perp$ at $y$, as in Figure \ref{geogebra figure}. Let $w_a$ denote the width of $\kappa_a$ in the direction of $\ell_{yx}^\perp$.
\begin{figure}\caption{The case when $\ell_{xy}^\perp$ is a support line for $\kappa$ in Lemma \ref{width lower bound for curves}.}
\label{geogebra figure}
\begin{center}
\includegraphics[scale=0.6]{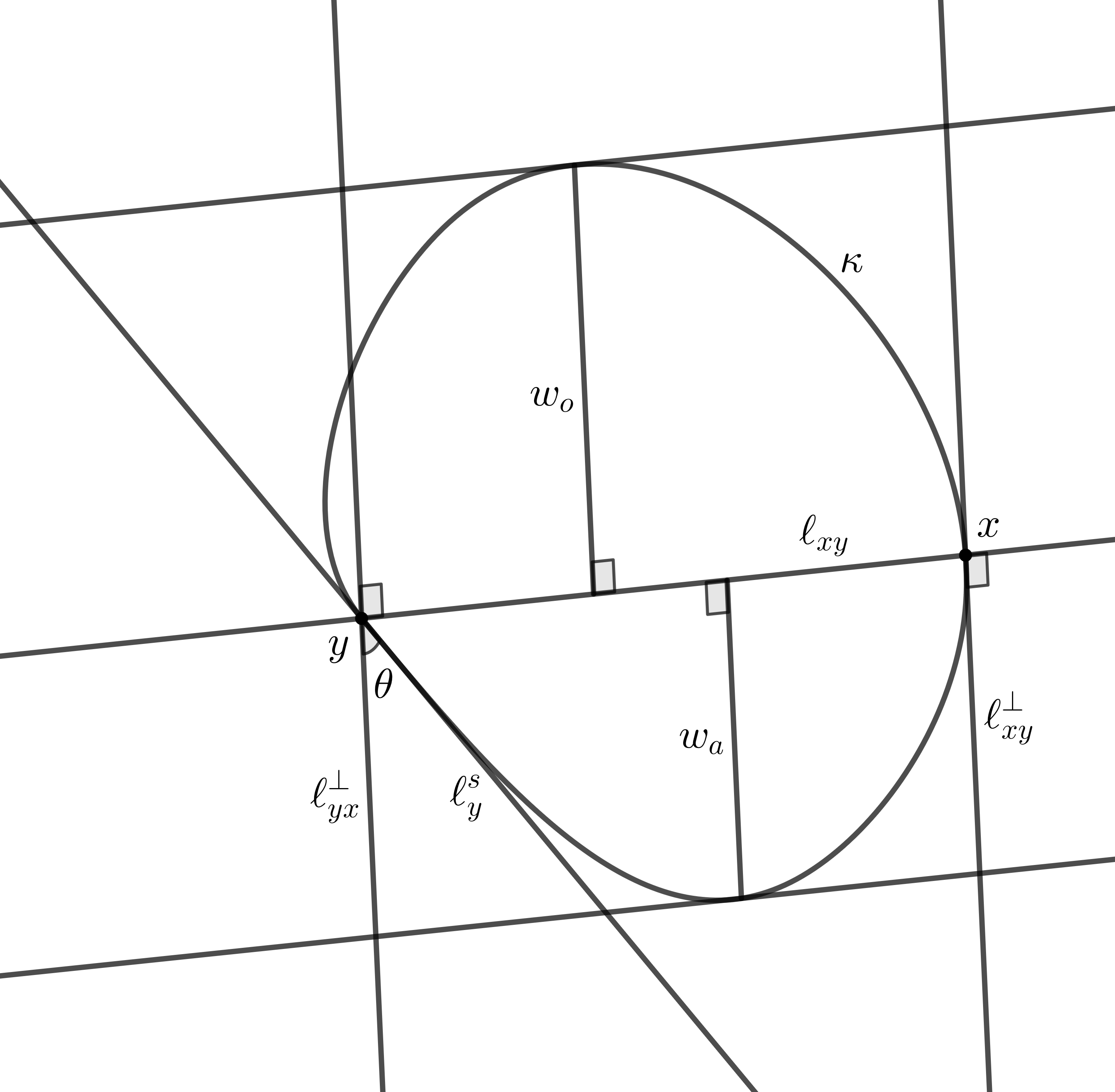}
\end{center}
\end{figure}
Note that $\ell_{xy}$ is a support line for $\kappa_a$. Since $\ell_y^s$ and $\ell_{xy}^\perp$ are also support lines for $\kappa_a$, $\kappa_a$ is contained within the convex region bounded by $\ell_y^s$, $\ell_{xy}^\perp$, $\ell_{xy}$, and the other support line to $\kappa_a$ parallel to $\ell_{xy}$. Thus,
\[
L(\kappa_a) \leq 2|y - x| + w_a + w_a \sec \theta - w_a \tan \theta.
\]
It follows that
\begin{align*}
d_\kappa (x,y) &\leq L(\kappa_a) - |y - x|\\
&\leq w_a(1 + \sec \theta - \tan \theta) + |y - x|\\
&\leq w_a(1 + \sec \theta - \tan \theta) + \frac{d_\kappa(x,y)}{\rho}.
\end{align*}
Solving the above inequality for $w_a$ shows that
\begin{equation}\label{acute inequality}
w_a \geq \frac{\rho - 1}{\rho(1 + \sec \theta - \tan \theta)} d_\kappa (x,y).
\end{equation}
Let $\kappa_o$ denote the convex curve formed from $\overline{xy}$ and the other arc of $\kappa \setminus \{ x, y \}$, and let $w_o$ denote the width of $\kappa_o$ in the direction of $\ell_{yx}^\perp$. A similar argument shows that
\begin{equation}\label{obtuse inequality}
w_o \geq \frac{\rho - 1}{\rho(1 + \sec \theta + \tan \theta)} d_\kappa (x,y).
\end{equation}
Because $w_{\ell_{xy}^\perp} (\kappa) = w_a + w_o$ and
\[
\frac{1}{1 + \sec \theta - \tan \theta} + \frac{1}{1 + \sec \theta + \tan \theta} = 1,
\]
the result follows from \eqref{acute inequality} and \eqref{obtuse inequality}. \end{proof}
\noindent When $\ell_{yx}^\perp$ is also a support line for $\kappa$, the proof of Lemma \ref{width lower bound for curves} yields the following stronger result.

\begin{corollary}\label{width lower bound for curves 2}
Let $\kappa$ be a closed, convex plane curve. Suppose distinct $x,y \in \kappa$ satisfy $d_\kappa (x,y) \geq \rho|y - x|$ for some $\rho > 1$. If $\ell_{xy}^\perp$ and $\ell_{yx}^\perp$ are both support lines for $\kappa$, then the two support lines for $\kappa$ parallel to $\ell_{xy}$ are both a distance of at least $\big( \frac{\rho - 1}{2\rho} \big) d_\kappa (x,y)$ from $\ell_{xy}$.
\end{corollary}

\begin{proof}
Because $\ell_{yx}^\perp$ is a support line of $\kappa$, one can take $\ell_y^s$ to be $\ell_{yx}^\perp$ and, consequently, $\theta$ to be $0$. The claim then follows from \eqref{acute inequality} and \eqref{obtuse inequality}.
\end{proof}

\noindent The next lemma is a special case of a result of Chakerian.

\begin{lemma}[\cite{Chakerian1967}, Theorem 2]\label{chakerian}
Let $\Omega$ be a compact, convex domain in $\real^{n+1}$. Let $H^n$ be a hyperplane in $\real^{n+1}$ and $\phi_H$ orthogonal projection onto $H$. Then, for any line segment $\lambda$ in $\Omega$ orthogonal to $H$,
\[
\mathrm{Vol}(\Omega) \geq \frac{L(\lambda) \mathrm{Area}(\phi_H(\Omega))}{n+1}.
\]
\end{lemma}

\begin{proof}[Proof of Proposition \ref{intrinsic-extrinsic proposition 1}]
Let $M^n$ be a closed, convex hypersurface in $\real^{n+1}$, and suppose $x$ and $y$ are distinct points of $M$ that satisfy $d_M (x,y) \geq \rho |y - x|$ for some $\rho > 1$. Let $N_x$ be the hyperplane through $x$ that is orthogonal to the segment $\overline{xy}$, $H_x^s$ and $H_y^s$ supporting hyperplanes of $M$ at $x$ and $y$, respectively, and $\tilde{H}_x$ the hyperplane parallel to $H_y^s$ through $x$. Since $d_M (x,y) > |y - x|$, $H_x^s$ and $\tilde{H}_x$ both intersect the line $\ell_{xy}$ transversely.
\smallskip 

Let $P_x = N_x \cap H_x^s \cap \tilde{H}_x$. For any $z \neq x$ in $P_x$, denote by $\Pi_z$ the $2$-dimensional plane in $\real^{n+1}$ containing $\overline{xy}$ and $z$, and let $\kappa_z = M \cap \Pi_z$. Since
\[
d_{\kappa_z} (x,y) \geq d_M (x,y) \geq \rho|y - x|,
\]
it follows from Corollary \ref{width lower bound for curves 2} that both support lines to $\kappa_z$ parallel to $\ell_{xy}$ are at a distance of at least $\big( \frac{\rho - 1}{2\rho} \big) d_M (x,y)$ from $\overline{xy}$. Because this is true for every $z \neq x$, the orthogonal projection of $M$ onto $P_x$ contains a ball of radius $\big( \frac{\rho - 1}{2\rho} \big) d_M (x,y)$ around $x$. The proof splits into three cases, depending on the manner in which $N_x$, $H_x^s$, and $\tilde{H}_x$ intersect:\\

\indent (1) $N_x = H_x^s = \tilde{H}_x$\\
\indent (2) $N_x = H_x^s \neq \tilde{H}_x$ or $N_x = \tilde{H}_x \neq H_x^s$\\
\indent (3) $H_x^s \neq N_x \neq \tilde{H}_x$\\

\noindent \textit{Case (1):} If $N_x = H_x^s = \tilde{H}_x$, then $\dim(P_x) = n$, i.e., $P_x$ is a hyperplane. The orthogonal projection of $M$ onto $P_x$ contains an $n$-dimensional ball of radius $\big( \frac{\rho - 1}{2\rho} \big) d_M (x,y)$, and the preimage of that ball consists of two disjoint open sets, on each of which the projection is a homeomorphism that doesn't increase area. Thus,
\begin{equation}\label{case 1 bound}
\mathrm{Area}(M) > 2 \omega_n \Big[ \frac{\rho - 1}{2\rho} d_M (x,y) \Big]^n = \frac{\omega_n}{2^{n-1}} \Big( \frac{\rho - 1}{\rho} \Big)^n d_M (x,y)^n,
\end{equation}
with a strict inequality since it's not possible both for those two open sets to contain the full measure of $M$ and for the projection to preserve their areas.\\

\noindent \textit{Case (2):} Without loss of generality, by swapping $x$ and $y$ if necessary, one may suppose that $N_x = H_x^s \neq \tilde{H}_x$. Let $\Pi_*$ be the $2$-dimensional plane in $\real^{n+1}$ containing $\ell_{xy}$ and the line in $N_x$ orthogonal to $\tilde{H}_x$, and let $\kappa_* = M \cap \Pi_*$. Since
\[
d_{\kappa_*}(x,y) \geq d_M(x,y) \geq \rho|y-x|,
\]
it follows from Lemma \ref{width lower bound for curves} that the width of $\kappa_*$ in the direction orthogonal to $\ell_{xy}$ is at least $\big( \frac{\rho - 1}{\rho} \big) d_M(x,y)$ and, consequently, that the projection $\phi_{N_x}(M)$ contains a line segment orthogonal to $P_x$ at least that long. Since $\phi_{P_x}(M)$ contains an $(n-1)$-dimensional ball with radius half of that, it has $(n-1)$-dimensional measure at least $\frac{\omega_{n-1}}{2^{n-1}} \big( \frac{\rho - 1}{\rho} \big)^{n-1} d_M(x,y)^{n-1}$. Because orthogonal projection onto $P_x$ factors through $N_x$, Lemma \ref{chakerian} implies that
\[
\mathrm{Area}(\phi_{N_x}(M)) \geq \frac{\omega_{n-1}}{n 2^{n-1}} \Big( \frac{\rho - 1}{\rho} \Big)^n d_M(x,y)^n.
\]
Arguing as in the first case, one has that
\begin{equation}\label{case 2 bound}
\mathrm{Area}(M) > 2\mathrm{Area}(\phi_{N_x}(M)) = \frac{\omega_{n-1}}{n 2^{n-2}} \Big( \frac{\rho - 1}{\rho} \Big)^n d_M(x,y)^n.
\end{equation}

\noindent \textit{Case (3):} If $H_x^s \neq N_x \neq \tilde{H}_x$, then, as discussed above, an $(n-2)$-dimensional subspace $V$ of $P_x$ contains an $(n-2)$-dimensional ball of radius $\big( \frac{\rho - 1}{2\rho} \big) d_M(x,y)$ centered at $x$. Let $\Pi_{**}$ be the $2$-dimensional plane in $\real^{n+1}$ containing $\ell_{xy}$ and the orthogonal complement to $V$ in $N_x \cap H_x^s$, and let $\kappa_{**} = M \cap \Pi_{**}$. By Lemma \ref{width lower bound for curves}, the width of $\kappa_{**}$ in the direction orthogonal to $\ell_{xy}$ is at least $\big( \frac{\rho - 1}{\rho} \big) d_\kappa (x,y)$, so, as in the second case, it follows from Lemma \ref{chakerian} that the $(n-1)$-dimensional measure of the projection $\phi_{N_x \cap H_x^s} (M)$ is at least $\frac{\omega_{n-2}}{(n-1) 2^{n-2}} \big( \frac{\rho - 1}{\rho} \big)^{n-1} d_M(x,y)^{n-1}$.
\smallskip 

Let $\Pi_*$ be the $2$-dimensional plane in $\real^{n+1}$ containing $\ell_{xy}$ and the orthogonal complement to $N_x \cap H_x^s$ in $N_x$, and let $\kappa_* = M \cap \Pi_*$. By Lemma \ref{max width bound}, the width of $\kappa_*$ in the direction of either $\ell_{xy}$ or its orthogonal complement is at least $\frac{1}{2} d_\kappa (x,y)$. In the former case,
\[
\mathrm{Area}(\phi_{N_x}(M)) \geq \frac{\omega_{n-2}}{n(n-1) 2^{n-1}} \Big( \frac{\rho - 1}{\rho} \Big)^{n-1} d_M (x,y)^n;
\]
in the latter case, denoting by $W$ the hyperplane in $\real^{n+1}$ containing $\ell_{xy}$ and $N_x \cap H_x^s$, $\mathrm{Area}(\phi_W(M))$ satisfies the same lower bound. In either case,
\begin{equation}\label{case 3 bound}
\mathrm{Area}(M) > \frac{\omega_{n-2}}{n(n-1) 2^{n-2}} \Big( \frac{\rho - 1}{\rho} \Big)^{n-1} d_M (x,y)^n.
\end{equation}\\

Taken together, \eqref{case 1 bound}, \eqref{case 2 bound}, and \eqref{case 3 bound} show that
\[
\mathrm{Area}(M) > \frac{m_n(\rho)}{2^{n-2}} \Big( \frac{\rho - 1}{\rho} \Big)^{n-1} d_M(x,y)^n
\]
for
\begin{equation}\label{formula for m}
\begin{aligned}
m_n(\rho) &= \min \Big\{ \frac{\omega_n}{2} \Big( \frac{\rho - 1}{\rho} \Big), \frac{\omega_{n-1}}{n} \Big( \frac{\rho - 1}{\rho} \Big), \frac{\omega_{n-2}}{n(n-1)} \Big\}\\
&= \min \Big\{ \frac{\omega_{n-1}}{n} \Big( \frac{\rho - 1}{\rho} \Big), \frac{\omega_{n-2}}{n(n-1)} \Big\},
\end{aligned}
\end{equation}
where the final equality follows from the fact that $2\omega_{n-1} \leq n\omega_n$, a special case of an inequality due to Alzer \cite[Theorem 3.6]{Alzer2008} and Klain--Rota \cite[Proposition 3.2]{KlainRota1997} that is also proved in Lemma \ref{beta bound}. Thus, the result holds for
\[
\mathcal{I}_n(\rho) = \frac{m_n(\rho)}{2^{n-2}} \Big( \frac{\rho - 1}{\rho} \Big)^{n-1}.
\]
The condition on the range of $\mathcal{I}_n$ is established in Remark \ref{formula for I}. \end{proof}

\begin{remark}\label{formula for I}
The two expressions, $\frac{\omega_{n-1}}{n} \big( \frac{\rho - 1}{\rho} \big)$ and $\frac{\omega_{n-2}}{n(n-1)}$, that appear in \eqref{formula for m} agree at
\begin{align}\label{rho_n_formula}
\displaystyle \rho_n = \frac{(n-1) \omega_{n-1}}{(n-1) \omega_{n-1} - \omega_{n-2}}.
\end{align}
The proof of Proposition \ref{intrinsic-extrinsic proposition 1} then gives the following explicit formula for $\mathcal{I}_n$:
\begin{equation}\label{explicit formula for I}
\mathcal{I}_n(\rho) = \left\{ \begin{array}{ccc} \frac{\omega_{n-1}}{n 2^{n-2}} \big( \frac{\rho - 1}{\rho} \big)^n & \textrm{if} & 1 < \rho \leq \rho_n, \\ \frac{\omega_{n-2}}{n(n-1) 2^{n-2}} \big( \frac{\rho - 1}{\rho} \big)^{n-1} & \textrm{if} & \rho > \rho_n. \end{array} \right.
\end{equation}
For each $\rho > 1$ and $n \geq 2$, there is, up to isometry, a unique right circular cylinder in $\real^{n+1}$ of height $\frac{1}{\rho}$ for which the intrinsic distance between the center points of the base and top is $1$, as depicted in Figure \ref{cylinder} for $n=2$. Because the extrinsic distance between these points is $\frac{1}{\rho}$, the ratio between their intrinsic and extrinsic distances is $\rho$, and the area of this cylinder, which we denote by $\overline{\mathcal{I}}_n(\rho)$, gives an upper bound for the optimal result in Proposition \ref{intrinsic-extrinsic proposition 1}. 
\begin{figure}\caption{A right circular cylinder whose area bounds the optimal constant in Proposition \ref{intrinsic-extrinsic proposition 1} from above, cf. Remark \ref{formula for I}.}
\label{cylinder}
\begin{center}
\includegraphics[scale=0.15]{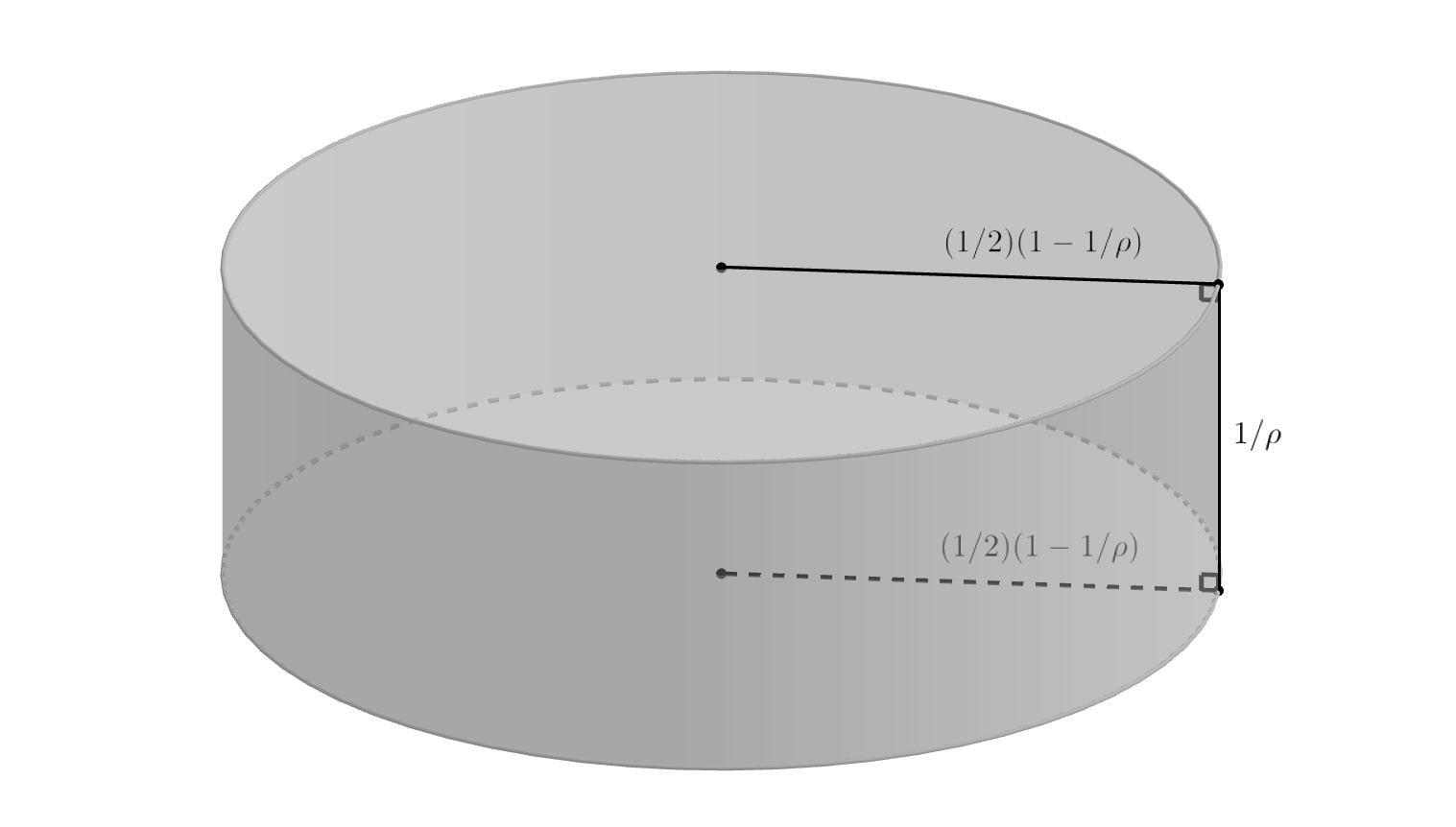}
\end{center}
\end{figure}
An elementary calculation shows that
\[
\overline{\mathcal{I}}_n(\rho) = \frac{\omega_n}{2^{n-1}} \Big( \frac{\rho - 1}{\rho} \Big)^{n-1} \Big( 1 + \frac{n-1}{\rho} \Big).
\]
The optimal result in Proposition \ref{intrinsic-extrinsic proposition 1} must go to zero as $\rho \to 1 +$, since $\overline{\mathcal{I}}_n (\rho)$ does. As $\rho \to \infty$, $\overline{\mathcal{I}}_n(\rho)$ increases to $\frac{\omega_n}{2^{n-1}}$, bounding the range of $\mathcal{I}_n$. With $\mathcal{I}_n$ as in \eqref{explicit formula for I}, for $\rho \geq \rho_n$, 
\[
\frac{\overline{\mathcal{I}}_n(\rho)}{\mathcal{I}_n(\rho)} = (n-1)\pi + \frac{(n-1)^2 \pi}{\rho}.
\]
Thus, on that portion of its domain, the function $\mathcal{I}_n$ in \eqref{explicit formula for I} is suboptimal by at most a factor of $n(n-1)\pi$ and, asymptotically, by at most a factor of $(n-1)\pi$.
\end{remark}

A long-standing conjecture of Alexandrov states that, among convex surfaces in $\real^3$ with intrinsic diameter $d$, the unique surface of maximal area is the doubled disk of radius $\frac{d}{2}$ (see \cite{Alexandrov1955,Ghomi2017,GrovePetersen1991}). An optimal result in Proposition \ref{intrinsic-extrinsic proposition 1} could provide a converse to Alexandrov's conjecture. When the hyperplanes $\ell_{xy}^\perp$ and $\ell_{yx}^\perp$ orthogonal to $\ell_{xy}$ at $x$ and $y$ in the proof of Proposition \ref{intrinsic-extrinsic proposition 1} are support planes, as in case (1), the proof gives a stronger result, which is asymptotically optimal in light of the bound given by $\overline{\mathcal{I}}_n(\rho)$ above. We record this in the following result.  

\begin{corollary}\label{asymptotically optimal cases}
Let $M^n$ be a closed, convex hypersurface in $\real^{n+1}$. Suppose there exist distinct $x,y \in M$ such that $\frac{d_M(x,y)}{|y-x|} \geq \rho$ and, also, that the hyperplanes $\ell_{xy}^\perp$ and $\ell_{yx}^\perp$ orthogonal to $\ell_{xy}$ at $x$ and $y$, respectively, are support planes for $M$. Let $\mathcal{I}^*_n(\rho) = \frac{\omega_n}{2^{n-1}} \big( \frac{\rho - 1}{\rho} \big)^n$, as in \eqref{case 1 bound}. Then,
\[
\mathrm{Area}(M) > \mathcal{I}_n^* (\rho) d_M(x,y)^n.
\]
\end{corollary}


\section{When intrinsic displacement is comparable to extrinsic displacement}
\label{int comp to ext}


The goal of this section is to prove a lower bound for the volume enclosed by a convex hypersurface $M^n$ in $\real^{n+1}$, in Proposition \ref{enclosed volume bound}, and an associated lower bound for the area of $M$ in Corollary \ref{intrinsic-extrinsic proposition 2}. Corollary \ref{intrinsic-extrinsic proposition 2} plays a role complementary to that of Proposition \ref{intrinsic-extrinsic proposition 1} in the proof of Theorem \ref{main theorem}. For example, Corollary \ref{intrinsic-extrinsic proposition 2} implies that a surface shaped like a long, narrow cigar with small area does not admit a fixed-point-free map that takes each point far from itself. 

\begin{proposition}\label{enclosed volume bound}
Let $M^n$ be a closed, convex hypersurface in $\real^{n+1}$ enclosing a domain $\Omega^{n+1}$ of volume $\mathrm{Vol}(\Omega)$. Suppose $\alpha$ is a fixed-point-free map from $M$ to itself, let $\rho(\alpha) = {\displaystyle \max_{x \in M}} \frac{d_M(x,\alpha(x))}{|\alpha(x) - x|}$, and let $K_{n+1} = \frac{2}{\sqrt{3}(n+1)!}$. Then,
\[
\mathrm{Vol}(\Omega) > \Big( \frac{K_{n+1}}{\rho(\alpha)^{n+1}} \Big) \mu(\alpha)^{n+1}. 
\]
\end{proposition}

\noindent By the isoperimetric inequality, this implies a lower bound for the area of $M$.

\begin{corollary}\label{intrinsic-extrinsic proposition 2}
Let $M^n$ be a closed, convex hypersurface in $\real^{n+1}$. If $\alpha$ is a fixed-point-free map from $M$ to itself, then
\[
\mathrm{Area}(M) > \mathcal{J}_n(\rho(\alpha)) \mu(\alpha)^n,
\]
where $\mathcal{J}_n(\rho) = \sigma_n \big( \frac{K_{n+1}}{\omega_{n+1}} \big)^{\frac{n}{n+1}} \frac{1}{\rho^n}$, $\sigma_n$ is the area of the unit sphere in $\real^{n+1}$, and $\rho(\alpha)$ and $K_{n+1}$ are as in Proposition \ref{enclosed volume bound}.
\end{corollary}

The core idea of the proof of Proposition \ref{enclosed volume bound} is as follows: if $M^n$ is a closed, convex hypersurface in $\real^{n+1}$ and $\alpha$ is a fixed-point-free map from $M$ to itself whose intrinsic displacement is never too large relative to its extrinsic displacement, the region enclosed by $M$ must be reasonably wide in every direction relative to the intrinsic displacement of $\alpha$. This bounds its volume below by a result of P\'{a}l \cite{Pal1921} and Firey \cite{Firey1965} quoted in Theorem \ref{pal-firey}. The continuity of $\alpha$ enters via Lemma \ref{chordal gauss map} below, which is based on the following result, drawn from work of Hirsch and Milnor.

\begin{lemma}[cf. {\cite[Lemma 3]{HirschMilnor1964}}]\label{hirsch-milnor}
Let $\alpha_0$ and $\alpha_1$ be fixed-point-free maps from $S^n$ to itself. Then, there is a homotopy $\Phi : [0,1] \times S^n \to S^n$ from $\Phi_0 = \alpha_0$ to $\Phi_1 = \alpha_1$ such that, for all $t \in [0,1]$, $\Phi_t : S^n \to S^n$ has no fixed points.
\end{lemma}

\begin{proof} This follows from the fact that any fixed-point-free map from $S^n$ to itself is homotopic to the usual antipodal map through fixed-point-free maps, which may be seen in the proof of the Hirsch--Milnor result cited above. \end{proof}

\begin{lemma}\label{chordal gauss map}
Let $M^n$ be a closed, convex hypersurface in $\real^{n+1}$ and $\alpha$ a fixed-point-free map from $M$ to itself. We define the \textbf{chordal Gauss map} of $\alpha$ to be the map $\Psi_\alpha : M \to S^n$ given by
\[
\Psi_\alpha(x) = \frac{\alpha(x) - x}{|\alpha(x) - x|}.
\]
Then, with $M$ and $S^n$ oriented as the boundaries of the domains they enclose, $\Psi_\alpha$ has degree $(-1)^{n+1}$. Consequently, $\Psi_\alpha$ is surjective.
\end{lemma}

\begin{proof} If $\alpha_{0},\alpha_{1}$ are fixed-point-free maps of $M$, a homotopy from $\alpha_{0}$ to $\alpha_{1}$ through fixed-point-free maps gives a homotopy from $\Psi_{\alpha_{0}}$ to $\Psi_{\alpha_{1}}$. Lemma \ref{hirsch-milnor} therefore implies that the chordal Gauss maps of all fixed-point-free maps of $M$ belong to the same homotopy class, so it suffices to check the degree of any of them. For any point $p$ in the interior of the region bounded by $M$, the map that takes each $x \in M$ to the other point at which $\ell_{xp}$ intersects $M$ is a fixed-point-free involution whose chordal Gauss map has degree $(-1)^{n+1}$. \end{proof}

\begin{remark}\label{Yang discussion}
If the map $\alpha$ in Lemma \ref{chordal gauss map} is an involution, as in Theorem \ref{berger antipodal bound}, there is an alternate proof of the surjectivity of $\Psi_{\alpha}$ based on the Borsuk--Ulam theorem: given $u \in S^n$ and $x \in M$ such that $x$ and $\alpha(x)$ have the same image under orthogonal projection onto the hyperplane orthogonal to $u$, the chord from $x$ to $\alpha(x)$ is parallel to $u$, and either $\Psi_{\alpha}(x) = u$ or $\Psi_{\alpha}(\alpha(x)) = u$. For involutions that are conjugate to the antipodal map in the homeomorphism group of the sphere, the Borsuk--Ulam theorem gives such an $x \in M$ for all $u \in S^n$, but, to the best of the authors' knowledge, it is not known whether all fixed-point-free involutions of spheres are of this type. Brouwer \cite{Brouwer1919} and Ker\'{e}kj\'{a}rt\'{o} \cite{Kerekjarto1919} proved that this is the case in dimension $2$, and Livesay \cite{Livesay1960} proved that this is the case in dimension $3$; however, Hirsch and Milnor \cite{HirschMilnor1964} showed that fixed-point-free involutions of spheres need not be smoothly or piecewise-linearly standard in dimensions $n \geq 5$, and work of Cappell--Shaneson \cite{CappellShaneson1976} and Akbulut--Kirby \cite{AkbulutKirby1979} implies that fixed-point-free involutions of the standard smooth structure on the $4$-sphere need not be smoothly standard. In all dimensions, however, Hirsch and Milnor's result in \cite[Lemma 3]{HirschMilnor1964} implies that any quotient of the sphere by a fixed-point-free involution is homotopy equivalent to real projective space, which implies by Yang's work in \cite{Yang1954} that the involution satisfies the conclusion of the Borsuk--Ulam theorem. Note that Berger's result in Theorem \ref{berger antipodal bound} as stated in \cite{Berger1977} applies to involutions that are conjugate to the antipodal map by homeomorphisms, and that this is the case for all involutions of the $2$-sphere by Brouwer and Ker\'{e}kj\'{a}rt\'{o}'s result cited above, so Theorem \ref{berger antipodal bound} applies to all such maps. The maps in the examples constructed by Croke in \cite{Croke2008} are conjugate to the standard antipodal map by diffeomorphisms, so Croke's result forecloses a lower bound for volume in terms of displacement as in Theorem \ref{berger antipodal bound} for any family of maps containing the standard antipodal map of the sphere of dimension $n \geq 3$. 
\end{remark}


\begin{theorem}[P\'{a}l \cite{Pal1921} for $d = 2$; Firey \cite{Firey1965} for $d \geq 3$]\label{pal-firey}
Let $\Omega$ be a compact, convex set in $\real^d$, $d \geq 2$, with minimum width $\underline{w}(\Omega)$. Then, $\mathrm{Vol}(\Omega) \geq \frac{2}{\sqrt{3} d!} \underline{w}(\Omega)^d$. If $d = 2$, then equality occurs if and only if $\Omega$ is an equilateral triangle.
\end{theorem}

\begin{remark}\label{bezdek bound}
The sharp version of Theorem \ref{pal-firey} is known only for $d = 2$, but, although the optimal constant and domain are not known for $d \geq 3$, it is known that the ball is not optimal: a cone of height $1$ over a $(d-1)$-disk of radius $\frac{1}{\sqrt{3}}$ contains a unit-length line segment in every direction and has smaller volume than the $d$-ball of radius $\frac{1}{2}$. Recent work of Arman, Bondarenko, Nazarov, Prymak, and Radchenko \cite{ArmanBondarenkoNazarovPrymakRadchenko2025} implies that, in sufficiently high dimensions, the optimal constant in Theorem \ref{pal-firey} is less than $\frac{\omega_d}{2^d} 0.9^d$, where $\frac{\omega_d}{2^d}$ is the constant given by the ball. On the other hand, Bezdek \cite[Theorem 6.2]{Bezdek2013} outlines a proof that, for $d \geq 3$, $\mathrm{Vol}(\Omega) \geq F_d \,\underline{w}(\Omega)^d$ for an explicit $F_d > \frac{2}{\sqrt{3}d!}$; thus, equality in Theorem \ref{pal-firey} occurs only when $d = 2$ and $\Omega$ is an equilateral triangle. We discuss this in more detail in Remark \ref{bezdek remark}. 
\end{remark}

\begin{proof}[Proof of Proposition \ref{enclosed volume bound}] Let $M^n$ be a closed, convex hypersurface in $\real^{n+1}$ and $\alpha$ a fixed-point-free map from $M$ to itself. Set $\rho(\alpha) = \displaystyle \max_{x \in M} \frac{d_M(x,\alpha(x))}{|\alpha(x) - x|}$. Then, for all $x \in M$,
\[
|\alpha(x) - x| \geq \frac{d_M(x,\alpha(x))}{\rho(\alpha)} \geq \frac{\mu(\alpha)}{\rho(\alpha)}.
\]
The surjectivity of the chordal Gauss map $\Psi_\alpha$ in Lemma \ref{chordal gauss map} therefore implies the minimum width of $\Omega$ is at least $\frac{\mu(\alpha)}{\rho(\alpha)}$. The lower volume bound then follows from Theorem \ref{pal-firey}. As discussed in Remark \ref{bezdek bound}, equality is impossible for $n \geq 2$, and equality for $n=1$ would imply that $M$ bounds an equilateral triangle. Any fixed-point-free map of such an $M$ must have $\rho(\alpha) \geq \sqrt{3}$: if $m$ is the midpoint of a side of $M$ and $p$ is the vertex opposite $m$, Lemma \ref{chordal gauss map} implies there is $x \in M$ with $\Psi_\alpha(x)$ parallel to the altitude of $M$ from $m$ to $p$, which must lie on the same side of $M$ as $m$. If $D$ is the distance from $x$ to a nearest vertex of $M$, then $d_M(x,\alpha(x)) = 3D$ and $|\alpha(x) - x| = \sqrt{3}D$, so $\frac{d_M(x,\alpha(x))}{|\alpha(x) - x|} = \sqrt{3}$. Equality would therefore imply $\mu(\alpha) \geq \frac{3}{2} \mathcal{L}$, where $\mathcal{L}$ is the length of the sides of $M$; however, $\mu(\alpha)$ is bounded above by $\frac{3}{2} \mathcal{L}$, with equality only for the map sending each $x \in M$ to the unique point at intrinsic distance $\frac{3}{2} \mathcal{L}$ from $x$. A direct calculation shows this map has $\rho(\alpha) = 2$, attained at points at distance $\frac{1}{4}\mathcal{L}$ from the vertices of $M$, so equality is impossible. \end{proof}

\begin{proof}[Proof of Corollary \ref{intrinsic-extrinsic proposition 2}] 
Let $\Omega$ denote the domain enclosed by $M$. By the isoperimetric inequality \cite{BuragoZalgaller1988} and Proposition \ref{enclosed volume bound}, $\mathrm{Area}(M) \geq \sigma_n \big( \frac{\mathrm{Vol}(\Omega)}{\omega_{n+1}} \big)^\frac{n}{n+1} > \mathcal{J}_n(\rho(\alpha)) \mu(\alpha)^n$. \end{proof}

\begin{remark}\label{pal_firey_bezdek_remark}
For any improved constant $F_d > \frac{2}{\sqrt{3} d!}$ in the conclusion of Theorem \ref{pal-firey}, one obtains a corresponding improvement to Proposition \ref{enclosed volume bound} and an improved rule for $\mathcal{J}_n$ in Corollary \ref{intrinsic-extrinsic proposition 2}, with $F_{n+1}$ in place of $K_{n+1}$. The proof of Proposition \ref{area bound proposition} then gives a stronger lower bound for the value of the constant in Theorem \ref{main theorem}, as discussed in Appendix \ref{asymptotics_appendix}. Whatever the optimal lower bound for $\mathrm{Vol}(\Omega)$ in terms of $\rho(\alpha)$ and $\mu(\alpha)$ as in Proposition \ref{enclosed volume bound}, the example given by reflection about the center of a right circular cylinder as in Remark \ref{formula for I} implies that it must go to zero as $\rho \to \infty$. Even with an optimal result in Theorem \ref{pal-firey}, it might be difficult to obtain an optimal result in Corollary \ref{intrinsic-extrinsic proposition 2}, and then in Theorem \ref{main theorem}, via the proof given here, because the optimal domain in the isoperimetric inequality is not optimal in Theorem \ref{pal-firey}, as discussed in Remark \ref{bezdek bound}. However, for convex domains, the isoperimetric inequality admits several quantitative refinements, discussed in \cite[Section 3]{Osserman1979}. Paired with information about the geometry of optimal or nearly optimal domains in Theorem \ref{pal-firey}, or, more broadly, a quantitative version of Theorem \ref{pal-firey} in the spirit of the results in \cite{Osserman1979}, these might allow one to strengthen the conclusion of Corollary \ref{intrinsic-extrinsic proposition 2}. Recent work of Lucardesi and Zucco \cite{LucardesiZucco2025} gives several quantitative results of this type for the $2$-dimensional case of Theorem \ref{pal-firey}. 
\end{remark}


\section{Displacement, area, and mean width in convex hypersurfaces}
\label{main results}


The goals of this section are to prove the following result, which implies Theorem \ref{main theorem}, and to prove Theorem \ref{mean width}. Note that the conclusion of Theorem \ref{main theorem} is immediate if $\alpha$ has a fixed point.

\begin{proposition}\label{area bound proposition}
For each natural number $n$, there exist $h_n > 0$ and a continuous function $\mathcal{B}_n : [1,\infty) \to [h_n,\infty)$ such that if $M^n$ is a closed, convex hypersurface in $\real^{n+1}$, $\alpha$ is a fixed-point-free map from $M$ to itself, and $\rho(\alpha)$ is the maximum ratio of intrinsic to extrinsic displacements of $\alpha$, as in Proposition \ref{enclosed volume bound}, then $\mathrm{Area}(M) > \mathcal{B}_n( \rho(\alpha) ) \mu(\alpha)^n$.
\end{proposition}

\begin{proof} Let $x \in M$ be a point such that $\rho(\alpha) = \frac{d_M(x,\alpha(x))}{|\alpha(x) - x|}$. By Proposition \ref{intrinsic-extrinsic proposition 1},
\[
\mathrm{Area}(M) > \mathcal{I}_n(\rho(\alpha)) d_M(x,\alpha(x))^n \geq \mathcal{I}_n(\rho(\alpha)) \mu(\alpha)^n.
\]
At the same time, by Corollary \ref{intrinsic-extrinsic proposition 2},
\[
\mathrm{Area}(M) > \mathcal{J}_n(\rho(\alpha)) \mu(\alpha)^n.
\]
Let $\mathcal{B}_n(\rho) = \max \{ \mathcal{I}_n(\rho), \mathcal{J}_n(\rho) \}$. Since $\mathcal{I}_n$ is increasing, $\mathcal{J}_n$ is decreasing, $\displaystyle \lim_{\rho \to 1^+} \mathcal{I}_n(\rho) = 0$, and $\displaystyle \lim_{\rho \to \infty} \mathcal{J}_n(\rho) = 0$, there exists $\rho_n^* > 1$ at which $\mathcal{I}_n$ and $\mathcal{J}_n$ assume their unique common value. The result follows, with $h_n$ given by $\mathcal{I}_n(\rho_n^*) = \mathcal{J}_n(\rho_n^*)$. \end{proof}

\begin{remark}\label{improvement_remark}
Because the optimal lower bound for volume as in Proposition \ref{enclosed volume bound} must go to $0$ as $\rho \to \infty$, as discussed in Remark \ref{pal_firey_bezdek_remark}, and the optimal result in Proposition \ref{intrinsic-extrinsic proposition 1} must go to $0$ as $\rho \to 1 +$, as explained in Remark \ref{formula for I}, the proof of Proposition \ref{area bound proposition} goes through as above with any improved versions of these results. The fact that $\mathcal{I}_{n} \to 0$ as $\rho \to 1 +$, together with the fact that $\mathcal{J}_{n}$ is well defined and positive when $\rho = 1$, also implies that the function $\mathcal{B}_{n}$ in Proposition \ref{area bound proposition} coincides with $\mathcal{J}_{n}$ for $\rho$ sufficiently close to $1$ and is therefore well defined and positive on the entire interval $[1,\infty)$. Because of this, it is not necessary to establish that the invariant $\rho(\alpha)$ appearing in the proof of Proposition \ref{area bound proposition} is strictly greater than $1$ in order to prove Theorem \ref{main theorem}. It seems likely that $\rho(\alpha)$ is strictly greater than $1$ for any fixed-point-free self-map $\alpha$ of a closed, convex hypersurface, but we do not see a simple proof of this in general. Because it is not necessary for the proof of our main results, we do not address this problem here. It is proved in Appendix \ref{int_ext_appendix} in several cases where it follows from known results. 
\end{remark}

We now establish Theorem \ref{mean width}. As mentioned in the introduction, we give two proofs of the theorem. They have a common first part, which establishes the inequality in (\ref{mean width eqn}) and the fact that, if equality holds, the hypersurface in question has constant width. This part of the proof also implies the theorem in its entirety for smooth hypersurfaces. We then give two arguments for the second part of the proof, which establishes the characterization of equality without assuming any regularity of the hypersurface in question. We begin with several preliminary results, in Lemmas \ref{constant width}, \ref{centrally symmetric and constant width}, \ref{isometric involution}, and \ref{filling lemma} and Theorem \ref{borisenko}. The first of these is a characterization of hypersurfaces of constant width.  For surfaces, this result is a classical theorem of Minkowski \cite{Minkowski1904}.

\begin{lemma}\label{constant width}
Let $M^n$ be a closed, convex hypersurface in $\real^{n+1}$, $n \geq 2$, and $Gr(2,n+1)$ the Grassmannian of unoriented $2$-planes through the origin in $\real^{n+1}$.  For $\Pi \in Gr(2,n+1)$, let $\phi_\Pi$ denote orthogonal projection onto $\Pi$, and suppose there is a constant $\xi$ such that $L(\partial \phi_\Pi (M)) = \xi$ for all $\Pi \in Gr(2,n+1)$. Then, $M$ has constant width $\frac{\xi}{\pi}$. 
\end{lemma}

\begin{proof} For each $u \in S^n$, let $\ell_u$ denote the line spanned by $u$. By Theorem \ref{crofton's formula},
\[
\xi = L(\partial \phi_\Pi (M)) = \frac{1}{2} \int_{\Pi \cap S^n} w_{\ell_u}(M) \ du 
\]
for all $\Pi \in Gr(2,n+1)$. Thus, $u \mapsto w_{\ell_u}(M) - \frac{\xi}{\pi}$ is an even function on $S^n$ that integrates to zero over every great circle. The invertibility of the spherical X-ray transform (i.e., the $1$-dimensional Radon transform) on even functions \cite[Theorem 3.1.7]{Helgason1999} implies that $w_{\ell_u}(M) - \frac{\xi}{\pi} = 0$ for all $u \in S^n$ and, therefore, that $M$ has constant width $\frac{\xi}{\pi}$. \end{proof}

\noindent The next result characterizes the round sphere among hypersurfaces of constant width by its central symmetry.

\begin{lemma}[cf. {\cite[p. 244]{Howard2006}}]\label{centrally symmetric and constant width}
Let $M^{n}$ be a closed, convex, centrally symmetric hypersurface of constant width in $\real^{n+1}$. Then, $M$ is a round sphere.
\end{lemma}

\begin{proof} Let $F: S^{n} \to \real$ be the support function of $M$ and $F_e(x) = \frac{[F(x) + F(-x)]}{2}$ and $F_o(x) = \frac{[F(x) - F(-x)]}{2}$ its even and odd parts, as in Martinez-Maure (e.g., \cite[pp. 590--592]{MartinezMaure1997}). The fact that $M$ has constant width is equivalent to the statement that $F_e$ is constant, and the fact that $M$ is centrally symmetric about a point $p \in \real^{n+1}$ is equivalent to the statement that $F_o$ is the dot product with $p$. Thus, if $R$ is the width of $M$, it follows that $F(x) = \frac{R}{2} + \langle p,x \rangle$, the support function of the sphere of radius $\frac{R}{2}$ about $p$. \end{proof}

\begin{lemma}\label{isometric involution}
Let $X$ be a subset of $\real^{n+1}$ endowed with the induced length metric $d$ and $\alpha$ a function from $X$ to itself such that, for all $x \in X$, the intersection of $X$ with any $2$-plane through $x$ and $\alpha(x)$ consists of two minimal geodesics of length $L$ from $x$ to $\alpha(x)$. Then, $\alpha$ is an involutive isometry. 
\end{lemma}

\begin{proof}Let $x \in X$, and slice $X$ by any $2$-plane containing $x, \alpha(x)$, and $\alpha(\alpha(x))$. If $\alpha(\alpha(x)) \neq x$, then $\alpha(\alpha(x))$ must lie on one of the two minimal geodesics of length $L$ between $x$ and $\alpha(x)$ in this slice; however, the proper subset of that geodesic connecting $\alpha(x)$ to $\alpha(\alpha(x))$ must also be a minimal geodesic of length $L$, a contradiction. Thus, $\alpha$ is an involution. For any $x,y \in X$, slice $X$ by two $2$-planes: one containing $x$, $y$, and $\alpha(x)$, and one containing $y$, $\alpha(x)$, and $\alpha(y)$. By the minimality of the geodesics in each slice, 
\begin{align*}
\displaystyle L = d(x,\alpha(x)) = d(x,y) + d(y,\alpha(x))
\end{align*}
and
\begin{align*}
\displaystyle L = d(y,\alpha(y)) = d(y, \alpha(x)) + d(\alpha(x), \alpha(y)).
\end{align*}
Therefore, $\displaystyle d(x,y) = L - d( y, \alpha(x)) = d(\alpha(x),\alpha(y))$, and $\alpha$ is an isometry. \end{proof}

\noindent Our next result implies that certain domains in convex surfaces satisfy the conclusion of the filling area conjecture. The situation it describes is depicted in Figure \ref{filling cap}.
\begin{figure}\caption{A convex surface that fills one of its planar slices, as in Lemma \ref{filling lemma}.}
\label{filling cap}
\begin{center}
\includegraphics[scale=0.3]{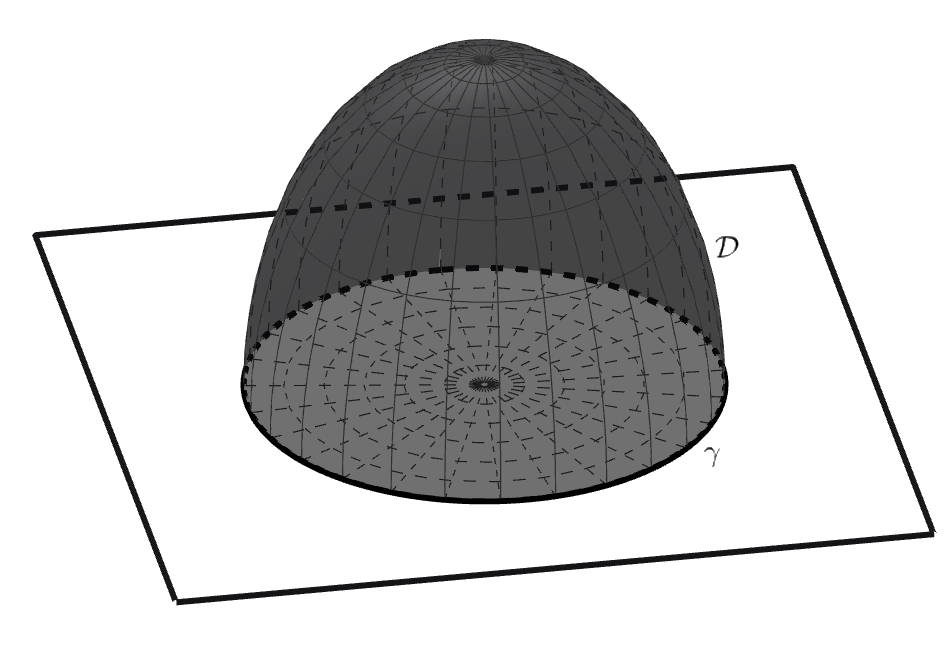}
\end{center}
\end{figure}
\begin{lemma}\label{filling lemma}
Let $M^2$ be a closed, convex surface in $\real^{3}$.  Let $\Pi$ be an affine $2$-plane in $\real^{3}$ that intersects $M$ in a closed curve $\gamma$, let $\mathcal{D}$ be a component of $M \setminus \gamma$, and suppose the inclusion $\gamma \hookrightarrow \overline{\mathcal{D}}$ is distance preserving, in that the distance between points $x,y \in \gamma$ in the length metric on $\overline{\mathcal{D}}$ is equal to the minimum length of the two components of $\gamma \setminus \lbrace x, y \rbrace$. Then, $\mathrm{Area}(\mathcal{D}) \geq \frac{L(\gamma)^{2}}{2\pi}$. 
\end{lemma}

\begin{proof} If $\mathcal{D}$ is smooth, then the result follows immediately from the filling inequality for Riemannian disks \cite{Gromov1983}. In general, let $M_{\varepsilon}$ be the outer parallel surface to $M$ at distance $\varepsilon$, let $\mathcal{D}_{\varepsilon}$ be the component of $M_{\varepsilon}$ that lies in the same component of $\real^{3} \setminus \Pi$ as $\mathcal{D}$, and let $\mathcal{D}_{\varepsilon}^{+}$ be a surface obtained by taking the union of $\mathcal{D}_{\varepsilon}$ with the annular region between $M$ and $M_{\varepsilon}$ in $\Pi$ and smoothing to obtain a smooth surface contained in the unbounded component of $\real^{3} \setminus M$ whose union with the domain bounded by $\gamma$ in $\Pi$ forms a closed, convex surface. The inclusion $\gamma \hookrightarrow \mathcal{D}_{\varepsilon}^{+}$ preserves the distance between $x,y \in \gamma$ because, for the bounded domain $\Omega$ enclosed by $M$, the nearest-point retraction $\real^{3} \setminus \Omega \to M$ is $1$-Lipschitz and, therefore, length non-increasing, and inclusion $\gamma \hookrightarrow \overline{\mathcal{D}}$ is distance preserving by assumption.  The filling inequality for disks then implies $\mathrm{Area}(\mathcal{D}_{\varepsilon}^{+}) \geq \frac{L(\gamma)^{2}}{2\pi}$. We can arrange for $\mathrm{Area}(\mathcal{D}_{\varepsilon}^{+})$ to converge to $\mathrm{Area}(\mathcal{D})$ as $\varepsilon \to 0$, for example by requiring $\mathcal{D}_{\varepsilon}^{+}$ to be within the $\epsilon$-neighborhood of $\mathcal{D}$, which implies the inequality for $\mathrm{Area}(\mathcal{D})$. \end{proof}

\noindent The next result is due to Borisenko and follows from the proof of \cite[Theorem 1.1\textquotesingle{}]{Borisenko2025}. Because it is slightly stronger than the result as stated in Borisenko's work, we explain in Remark \ref{borisenko remark} how to derive this conclusion from the proof of the result there.  

\begin{theorem}[Borisenko \cite{Borisenko2025}]\label{borisenko}
Let $M^n$ and $N^n$ be closed and convex hypersurfaces in $\real^{n+1}$, $n \geq 2$. If $\alpha : M \to N$ is an isometry, then $\alpha$ extends to a rigid motion of $\real^{n+1}$.
\end{theorem}

\begin{remark}\label{borisenko remark}
Although \cite[Theorem 1.1\textquotesingle{}]{Borisenko2025} states only that $M$ and $N$ are related by a rigid motion, the argument there establishes the result stated in Theorem \ref{borisenko}. Specifically, one may replace the first conclusion ``the hypersurfaces $F_1$ and $F_2$ are congruent'' in Lemma 1.2 of \cite{Borisenko2025} with ``the given isometry between $F_1$ and $F_2$ extends to a rigid motion.'' This is because the proof of the lemma, due to Sen'kin \cite[Lemma 3]{Senkin1972}, shows that $|\alpha(y) - \alpha(x)| = |y - x|$ for all $x,y \in M$.
\end{remark}

\begin{proof}[Proof of Theorem \ref{mean width}, Part One] Let $M^n$ be a closed, convex hypersurface in $\real^{n+1}$, $n \geq 2$, and $\alpha$ a map from $M$ to itself. If $\alpha$ has a fixed point, then the result is immediate, so we can suppose $\alpha$ is fixed-point free. Lemma \ref{chordal gauss map} implies that the chordal Gauss map $\Psi_\alpha$ of $\alpha$ is surjective. Fix any orthogonal unit vectors $u$ and $v$ in $\real^{n+1}$. By surjectivity, there exists $x \in M$ such that $\Psi_\alpha(x) = u$. Denote by $\Pi_{uv}$ the $2$-plane spanned by $u$ and $v$, and let $\Pi_{uv}^x$ be the affine $2$-plane parallel to $\Pi_{uv}$ through $x$. Then,
\begin{align}\label{mean width pf eqn 1}
\displaystyle \mu(\alpha) \leq d_M(x,\alpha(x)) \leq \frac{1}{2} L(\Pi_{uv}^x \cap M).
\end{align}	
Let $\phi_{\Pi_{uv}}$ denote orthogonal projection onto $\Pi_{uv}$. By Proposition \ref{crofton corollary},
\begin{align}\label{mean width pf eqn 2}
\displaystyle L(\Pi_{uv}^x \cap M) \leq L(\partial \phi_{\Pi_{uv}}(M)).
\end{align}
Therefore, by Theorem \ref{crofton's formula},
\begin{align}\label{mean width pf eqn 3}
\displaystyle \mu(\alpha) \leq \frac{\pi}{2} \int\limits_{Gr(1,\Pi_{uv})} w_\ell\,\overline{d\ell},
\end{align}
where $\overline{d\ell}$ is the unit-mass $SO(2)$ invariant measure on $Gr(1,\Pi_{uv})$. Note that the $SO(n+1)$-invariant measure $d\ell$ on $Gr(1,n+1)$ in Definition \ref{width definitions} is the push-forward of the $SO(n+1)$-invariant measure on the total space of the natural $Gr(1,\Pi_{uv})$-bundle over $Gr(2,n+1)$. Integrating the inequality in (\ref{mean width pf eqn 3}) over $Gr(2,n+1)$ shows that $\Xi(M) \geq \frac{2}{\pi} \mu(\alpha)$.
\smallskip 

Suppose equality holds. For all $x \in M$ and all unit vectors $v$ orthogonal to $u = \Psi_\alpha(x)$, equality in (\ref{mean width pf eqn 1}) and (\ref{mean width pf eqn 2}) implies that, for all such $x$, $u$, and $v$,
\begin{align}\label{mean width pf eqn 4}
\displaystyle\mu(\alpha) = d_M(x,\alpha(x)) = \frac{1}{2} L(\Pi_{uv}^x \cap M) = \frac{1}{2} L(\partial \phi_{\Pi_{uv}}(M)).
\end{align}
The inequality $\mu(\alpha) = \frac{1}{2} L(\Pi_{uv}^x \cap M)$ in (\ref{mean width pf eqn 4}) implies that all $2$-planes through $x$ and $\alpha(x)$ intersect $M$ in a pair of minimizing geodesics from $x$ to $\alpha(x)$.  If $M$ is smooth, then the theorem follows from the solution to the Blaschke conjecture \cite{Besse1978}. In general, the equality $\mu(\alpha) = \frac{1}{2} L(\partial \phi_{\Pi_{uv}}(M))$ in (\ref{mean width pf eqn 4}) implies, by Lemma \ref{constant width}, that $M$ has constant width $\frac{2}{\pi} \mu(\alpha)$. \end{proof}

\begin{proof}[Proof of Theorem \ref{mean width}, Part Two, First Version] For $x \in M$, $u = \Psi_{\alpha}(x)$, and any unit vector $v$ orthogonal to $u$ as above, equality in (\ref{mean width pf eqn 2}) implies by Proposition \ref{crofton corollary} that, under the natural identification between the parallel $2$-planes $\Pi_{uv}$ and $\Pi_{uv}^{x}$, 
\begin{align}\label{mean width pf eqn 5}
\displaystyle \partial \phi_{\Pi_{uv}}(M) = \Pi_{uv}^x \cap M.
\end{align}	
Moreover, $\Pi_{uv}^x$ is the only affine $2$-plane parallel to $\Pi_{uv}$ for which the equality in (\ref{mean width pf eqn 5}) holds, as otherwise the domain $\Omega$ enclosed by $M$ would contain a $3$-dimensional cylinder with base $\Omega \cap \Pi_{uv}^x$ whose width along some lines would be greater than its width along lines contained in $\Pi_{uv}$, contradicting the fact that $M$ has constant width. Because this holds for all $\Pi_{u'v'}^y$ with $u' = \Psi_\alpha(y)$ and $v'$ orthogonal to $u'$, every $y \in M$ with $\Psi_\alpha(y) \in \Pi_{uv}$ must belong to $\Pi_{uv}^x \cap M$. For any unit vector $w$ in $\Pi_{uv}$, there is exactly one line $\ell_w$ in $\Pi_{uv}^x$ parallel to $w$ that divides $\Pi_{uv}^x \cap M$ into two arcs of equal length. The equality $d_M(x,\alpha(x)) = \frac{1}{2} L(\Pi_{uv}^x \cap M)$ in (\ref{mean width pf eqn 4}), as applied to each $y \in \Pi_{uv}^x \cap M$, implies that any $y \in M$ with $\Psi_\alpha(y) = w$ must lie on $\ell_w$ and, therefore, that $\alpha$ must act by exchanging the two points at which $\ell_w$ meets $M$.  In particular, $\alpha$ maps $\Pi_{uv}^x \cap M$ to itself.  
\smallskip 

Because $\alpha$ maps $\Pi_{uv}^x \cap M$ to itself, and because, for every $y \in M$, all affine $2$-planes through $y$ and $\alpha(y)$ intersect $M$ in a pair of minimizing geodesics from $y$ to $\alpha(y)$, for all unit vectors $z$ orthogonal to $\Pi_{uv}$, the inclusion map $\Pi_{uv}^x \cap M \hookrightarrow \Pi_{uvz}^x \cap M$ does not decrease distance between points in $\Pi_{uv}^x \cap M$ with respect to the length metric on $\Pi_{uvz}^x \cap M$, where $\Pi_{uvz}$ is the linear subspace spanned by $u,v,z$ and $\Pi_{uvz}^{x}$ is the affine $3$-dimensional subspace parallel to $\Pi_{uvz}$ through $x$. By Lemma \ref{filling lemma}, this implies that the area of $\Pi_{uvz}^x \cap M$ is at least $\frac{4}{\pi} \mu(\alpha)^2$. By Proposition \ref{crofton corollary}, this implies the same lower bound for the area of $\partial \phi_{\Pi_{uvz}}(M)$, where $\phi_{\Pi_{uvz}}$ denotes orthogonal projection onto $\Pi_{uvz}$.
\smallskip 

On the other hand, because $M$ has constant width $\frac{2}{\pi} \mu(\alpha)$, $\partial \phi_{\Pi_{uvz}}(M)$ does as well. Minkowski's inequality implies that its area is bounded above by $\frac{4}{\pi} \mu(\alpha)^2$. For all linear subspaces $\Pi_{uvz}$, the upper bound for the area of $\partial \phi_{\Pi_{uvz}}(M)$ established above implies equality in Minkowski's inequality, which implies that $\partial \phi_{\Pi_{uvz}}(M)$ is a round sphere. This implies $M$ is a round sphere. The nonuniqueness of minimizing geodesics between $x$ and $\alpha(x)$ implies that $\alpha$ is the antipodal map. \end{proof}

\begin{proof}[Proof of Theorem \ref{mean width}, Part Two, Second Version] The second equality in (\ref{mean width pf eqn 4}) implies, by Lemma \ref{isometric involution}, that $\alpha$ is an involutive isometry of $M$ with respect to its intrinsic metric. Theorem \ref{borisenko} implies that $\alpha$ extends to an involutive isometry of $\real^{n+1}$.  By the classification of isometries of $\real^{n+1}$ and the fact that $\alpha$ is fixed-point free on $M$, $\alpha$ must be a reflection through a point. Because $M$ has constant width, Lemma \ref{centrally symmetric and constant width} implies that $M$ is a round sphere and $\alpha$ is the antipodal map. \end{proof}

\begin{remark}\label{final remark}
For odd $n$, the mean width of a smooth, convex hypersurface $M^n$ in $\real^{n+1}$ is an intrinsic invariant of the Riemannian metric on $M$, given by the integral of a functional of its curvature tensor: for $n = 1$, it is proportional to the length of $M$ by Theorem \ref{crofton's formula}, and for $n = 3$, it is proportional to the total scalar curvature of $M$. In general, it is equal to one of the Lipschitz--Killing curvatures of $M$ that give the coefficients in Weyl's tube formula, as discussed in \cite{Gray2004}. As written, Theorem \ref{mean width} cannot extend to a result about this invariant for all Riemannian metrics on the $n$-sphere---the $3$-sphere admits metrics with negative scalar curvature and, in fact, negative Ricci curvature \cite{GaoYau1986, Lohkamp1994}---but it seems natural to ask whether Theorem \ref{mean width} can be subsumed into a more general result of this type. For even $n$, the mean width of a convex hypersurface $M^n$ in $\real^{n+1}$ does not admit a representation in terms of the local intrinsic geometry of $M$, but it is proportional in all dimensions to the integral of the $(n-1)$-st elementary symmetric function of the principal curvatures of $M$ (cf. \cite{Gray2004,Santalo2004,Schneider1993}). Theorem \ref{mean width} therefore gives a sharp lower bound for several invariants of smooth convex hypersurfaces, including mean curvature, as in the following result.
\end{remark}

\begin{corollary}\label{mean curvature}
Let $M^n$ be a smooth, closed, convex hypersurface in $\real^{n+1}$, let $\alpha$ be a continuous map from $M$ into itself with minimal displacement $\mu(\alpha)$, as in Theorems \ref{main theorem} and \ref{mean width}, and let $\mathcal{H}$ be the mean curvature of $M$. Then, 
\begin{align*}
\displaystyle \displaystyle \int_{M} \mathcal{H}^{n-1} \ dx \geq P_{n} \mu(\alpha), 
\end{align*}
where $P_{n} = \frac{n^{n-1}\sigma_{n}}{\pi}$. For $n \geq 2$, equality holds if and only if $M$ is a round sphere and $\alpha$ is the antipodal map. 
\end{corollary}

\begin{proof} Let $\eta_{m}$ denote the $m$-th elementary symmetric function; that is, for $d \geq m$, 
\begin{align*}
\displaystyle \eta_{m}(x_{1}, \dots, x_{d}) = \sum\limits_{1 \leq i_{1}< \cdots < i_{m} \leq d} x_{i_{1}}x_{i_{2}}\cdots x_{i_{m}}. 
\end{align*}
Denote by $0 \leq k_{1} \leq k_{2} \leq \dots \leq k_{n}$ the principal curvatures of $M$. The identity for $\Xi(M)$ alluded to above states that 
\begin{align}\label{mean curv cor eqn}
\displaystyle \Xi(M) = \frac{2}{n\sigma_{n}} \int_{M} \eta_{n-1}(k_{1}, \cdots, k_{n}) \ dx. 
\end{align}
Together, Theorem \ref{mean width}, (\ref{mean curv cor eqn}), and Maclaurin's inequalities for symmetric functions \cite{HardyLittlewoodPolya1952} imply that, for $1 \leq m \leq n-1$, 
\begin{align*}
\displaystyle \frac{\sigma_n}{\pi} \binom{n}{m}^{\frac{n-1}{m}} \mu(\alpha) \leq \int_{M} \eta_{m}(k_{1}, \cdots, k_{n})^{\frac{n-1}{m}} \ dx.
\end{align*}
The case $m = 1$ gives the result. \end{proof}
\noindent For $n \geq 3$, the case $m=2$ in the proof of Corollary \ref{mean curvature} also gives a lower bound for the scalar curvature of smooth convex hypersurfaces. 


\appendix


\section{Asymptotic behavior of the constants in Theorem \ref{main theorem}}
\label{asymptotics_appendix}


The aim of this appendix is to describe the asymptotic behavior of our estimate for the constants $h_n$ in Theorem \ref{main theorem}. Let $\Gamma(z)$ denote the gamma function $\int_0^\infty t^{z-1} e^{-t} dt$. The volume of the unit $n$-ball and area of the unit $n$-sphere are $\omega_n = \frac{\pi^{\frac{n}{2}}}{\Gamma ( \frac{n}{2} + 1 )}$ and $\sigma_n = \frac{2\pi^{\frac{n+1}{2}}}{\Gamma (\frac{n+1}{2})}$, respectively, and are related by the identity $\sigma_{n+1} = 2\pi \omega_n$. The following inequality for the volumes of balls, which appears in the proof of Proposition \ref{intrinsic-extrinsic proposition 1}, can also be derived from work of Alzer and Klain--Rota.

\begin{lemma}[cf. {\cite[Theorem 3.6]{Alzer2008}}; {\cite[Proposition 3.2]{KlainRota1997}}]\label{beta bound}
For all natural numbers $n$, $2\omega_{n-1} \leq n\omega_n$.
\end{lemma}

\begin{proof} The desired inequality is equivalent to $2\Gamma(\frac{n+2}{2}) \leq n\sqrt{\pi}\Gamma(\frac{n+1}{2})$, which, because $\Gamma(\frac{1}{2}) = \sqrt{\pi}$, is equivalent to the inequality $B(\frac{n+1}{2},\frac{1}{2}) \geq \frac{2}{n}$ for the beta function $B(z_1,z_2) = \frac{\Gamma(z_1)\Gamma(z_2)}{\Gamma(z_1 + z_2)}$. By the identity $B(z_1,z_2) = 2\int_0^{\pi/2} (\sin \theta)^{2z_1 - 1} (\cos \theta)^{2z_2 - 1} d\theta$, this is equivalent to $\int_0^{\pi/2} \sin^n \theta \,d\theta \geq \frac{1}{n}$. Letting $I_n$ equal that integral, a classical reduction formula due to Wallis states that $I_{n+2} = \big( \frac{n+1}{n+2} \big) I_n$. Since $I_1 = 1$ and $I_2 = \frac{\pi}{4}$, the result follows by induction. \end{proof}

Let $\mathcal{I}_n$ be the function in (\ref{explicit formula for I}), and let $\mathcal{J}_n$ be as in Corollary \ref{intrinsic-extrinsic proposition 2}. Our estimate for the constant $h_n$ in Theorem \ref{main theorem} is the common value of $\mathcal{I}_n$ and $\mathcal{J}_n$ at the unique $\rho_n^* > 1$ for which $\mathcal{I}_n(\rho_n^*) = \mathcal{J}_n(\rho_n^*)$. Letting $\rho_n$ be as in (\ref{rho_n_formula}), we must determine whether $\rho_n^* \leq \rho_n$ or $\rho_n^* > \rho_n$ in order to solve for $\rho_n^*$. 

\begin{lemma}\label{intersection piece}
For each natural number $n$, let $\mathcal{I}_n$,  $\mathcal{J}_n$, and $\rho_n$ be as above, and let $\rho_n^*$ be the unique $\rho > 1$ with $\mathcal{I}_n(\rho_n^*) = \mathcal{J}_n(\rho_n^*)$. Let $A_n = (2^{n-1} \pi n)^{\frac{1}{n}} \big( \frac{K_{n+1}}{\omega_{n+1}} \big)^{\frac{1}{n+1}}$, with $K_{n+1} = \frac{2}{\sqrt{3}(n+1)!}$ as in Proposition \ref{enclosed volume bound}, and let $B_n = \frac{1}{(n-1) \frac{\omega_{n-1}}{\omega_{n-2}} - 1}$. Then, $\rho_n^* > \rho_n$ if and only if $A_n > B_n$.
\end{lemma}

\begin{proof} If $\rho_n^* \leq \rho_n$, then $\mathcal{I}_n(\rho_n^*) = \frac{\omega_{n-1}}{n 2^{n-2}} \big( \frac{\rho_n^* - 1}{\rho_n^*} \big)^n$, which implies $\rho_n^*$ satisfies the following identity: 
\begin{align}\label{intersection piece pf eqn 2}
\displaystyle \sigma_n \Big( \frac{K_{n+1}}{\omega_{n+1}} \Big)^{\frac{n}{n+1}} = \frac{\omega_{n-1}}{n 2^{n-2}} (\rho_n^* - 1)^n .
\end{align}	
By the identity $\frac{\sigma_n}{\omega_{n-1}} = 2\pi$, this is equivalent to $\rho_n^* = 1 + A_n$. We also have that
\begin{align*}
\displaystyle \rho_n = 1 + \frac{\omega_{n-2}}{(n-1)\omega_{n-1} - \omega_{n-2}} = 1 + B_n .
\end{align*}
Together, these identities imply that $A_n \leq B_n$. Conversely, the identity in \eqref{intersection piece pf eqn 2} is valid if one sets $\rho_n^* = 1 + A_n$. If $A_n \leq B_n$, then multiplying \eqref{intersection piece pf eqn 2} by $\big( \frac{1}{\rho_n^*} \big)^{n}$ gives $\mathcal{I}_n(\rho_n^*) = \mathcal{J}_n(\rho_n^*)$. \end{proof}

\noindent We have verified computationally that $A_n > B_n$, and therefore that $\rho_n^* > \rho_n$, for $n \leq 100,000$. We will show that this is the case for all sufficiently large $n$ in Lemma \ref{further asymptotics} below. In several calcuations, we use the following generalization of Stirling's approximation for the gamma function due to Binet.  

\begin{theorem}[\cite{Binet1839}]\label{binet_theorem}
For $z > 0$ and any natural number $N$, 
\begin{equation*}
\ln \Gamma(z+1) = z \ln z - z + \frac{1}{2} \ln z + \frac{1}{2} \ln(2\pi) + \sum_{k = 1}^N \frac{\beta_{2k}}{2k(2k-1)z^{2k-1}} + \epsilon(N,z), 
\end{equation*}
where $\beta_m$ is the $m$-th Bernoulli number and, for fixed $N$, the error term $\epsilon(N,z)$ is $O \big( \frac{1}{z^{2N + 1}} \big)$ as $z \to \infty$. In particular, $\Gamma(z + 1) = \sqrt{2\pi z} \big( \frac{z}{e} \big)^z \big[ 1 + O \big( \frac{1}{z} \big) \big]$. 
\end{theorem}
\noindent Theorem \ref{binet_theorem} leads to the following asymptotic estimate for the volume of the unit $n$-ball.

\begin{lemma}\label{ball asympototics}
For each fixed integer $m$, as $n \to \infty$, 
\begin{equation*}
\displaystyle \ln \omega_{n+m} = \frac{1}{2} \Big[ - n \ln n + \ln(2\pi e) n - (m+1) \ln n + m \ln (2\pi)  - \ln \pi  + O \Big( \frac{1}{n} \Big) \Big].
\end{equation*}
\end{lemma}

\begin{proof} For a fixed integer $m$ and $n > |m|$, a calculation via the Maclaurin series $-\sum_{k=1}^\infty \frac{x^k}{k}$ for $\ln(1 - x)$ for $|x| < 1$ establishes that
\begin{equation}\label{logarithm shift}
\ln(n + m) = \ln n + \sum_{k = 1}^\infty (-1)^{k - 1} \frac{m^k}{k n^k} = \ln n + \frac{m}{n} + O \Big( \frac{1}{n^2} \Big) .
\end{equation}
A similar calculation via geometric series establishes that  
\begin{equation}\label{reciprocal shift}
\frac{1}{n + m} = \sum_{k = 1}^\infty (-1)^{k-1} \frac{m^{k-1}}{n^k} = \frac{1}{n} - \frac{m}{n^2} + O \Big( \frac{1}{n^3} \Big).
\end{equation}
From the formula $\omega_{n+m} = \frac{\pi^{\frac{n+m}{2}}}{\Gamma \big( \frac{n+m}{2} + 1 \big)}$ and Theorem \ref{binet_theorem}, one has that
\begin{equation}\label{ball eqn}
\displaystyle \ln \omega_{n+m} = \frac{1}{2} \big[ -(n+m) \ln (n+m) + \ln(2\pi e) (n+m) - \ln (n+m) - \ln \pi + O \Big( \frac{1}{n} \Big) \Big].
\end{equation}
Together, \eqref{logarithm shift}, \eqref{reciprocal shift}, and \eqref{ball eqn} imply the result. \end{proof}

\begin{lemma}\label{asymptotics_prop_1}\label{further asymptotics}
For $A_n$ and $B_n$ as in Lemma \ref{intersection piece}, $\lim\limits_{n \to \infty} \frac{A_n}{B_n} = 2\sqrt{e}$.  
\end{lemma}

\begin{proof} From the definition of $A_n$, 
\begin{align}\label{further_asymp_pf_eqn_1}
\displaystyle \ln A_n = \Big( 1 - \frac{1}{n} \Big) \ln 2 + \frac{\ln \pi}{n} + \frac{\ln n}{n} + \frac{1}{n+1} (\ln K_{n+1} - \ln \omega_{n+1}),
\end{align}	
where $K_{n+1} = \frac{2}{\sqrt{3} (n+1)!}$ as above. Theorem \ref{binet_theorem}, \eqref{logarithm shift}, and the identity $(n+1)! = (n+1)\Gamma(n+1)$ imply that 
\begin{align*}
\displaystyle \ln K_{n+1} = - n \ln n + n - \frac{3}{2} \ln n + \frac{1}{2} \ln \Big( \frac{2}{3\pi} \Big) + O \Big( \frac{1}{n} \Big).
\end{align*} 
Combining this with Lemma \ref{ball asympototics} yields
\begin{align}\label{further_asymp_pf_eqn_5}
\displaystyle \ln K_{n+1} - \ln \omega_{n+1} = \frac{1}{2} \Big[ - n \ln n + \ln \Big( \frac{e}{2\pi} \Big) n - \ln n + \ln \Big( \frac{1}{3\pi} \Big) + O \Big( \frac{1}{n} \Big) \Big].
\end{align}
From this and the elementary fact that $\frac{n}{n+1} = 1 - \frac{1}{n} + O \big( \frac{1}{n^2} \big)$, one finds that
\begin{equation}\label{K minus ball}
\frac{n}{n+1} (\ln K_{n+1} - \ln \omega_{n+1}) = \frac{1}{2} \Big[ - n \ln n + \ln \Big( \frac{e}{2\pi} \Big) n + \ln \Big( \frac{2}{3e} \Big) + O \Big( \frac{\ln n}{n} \Big) \Big].
\end{equation}
Dividing (\ref{K minus ball}) by $n$ then gives
\begin{equation}\label{kmb2}
\displaystyle \frac{1}{n+1}(\ln K_{n+1} - \ln \omega_{n+1}) = \frac{1}{2} \Big[ -\ln n + \ln(\frac{e}{2\pi}) + O \Big( \frac{1}{n} \Big) \Big].	
\end{equation}
Together, \eqref{further_asymp_pf_eqn_1} and \eqref{kmb2} imply that
\begin{equation}\label{mfape_a}
\displaystyle \ln A_n = -\frac{1}{2} \ln n + \frac{1}{2} \ln \Big( \frac{2e}{\pi} \Big) + \frac{\ln n}{n} + O \Big( \frac{1}{n} \Big). 
\end{equation}
From the definition of $B_n$,
\begin{align*}
\displaystyle \ln B_{n+1} &= \ln \omega_{n-1} - \ln n - \ln \omega_n  - \ln \Big( 1 - \frac{\omega_{n-1}}{n\omega_n}  \Big).
\end{align*}
By the $m=-1$ and $m=0$ cases of Lemma \ref{ball asympototics}, we then have
\begin{align}\label{further_asymp_pf_eqn_2}
\displaystyle \ln B_{n+1} = -\frac{1}{2} \ln n - \frac{1}{2} \ln(2\pi) - \ln \Big( 1 - \frac{\omega_{n-1}}{n\omega_n}  \Big) + O \Big( \frac{1}{n} \Big).
\end{align}
By a calculation with the Maclaurin series for $\ln(1 - x)$ for $|x| < 1$, as in the proof of Lemma \ref{ball asympototics}, 
\begin{align}\label{further_asymp_pf_eqn_3}
\displaystyle \ln \Big( 1 - \frac{\omega_{n-1}}{n\omega_n}  \Big) = -\frac{\omega_{n-1}}{n\omega_n} + O \Big( \frac{1}{n^2} \Big). 
\end{align}
By the $m=-1$ and $m=0$ cases of Lemma \ref{ball asympototics} as above, 
\begin{align*}
\displaystyle \ln \Big( \frac{\omega_{n-1}}{n \omega_n} \Big) = \frac{1}{2} \ln \Big( \frac{1}{2\pi n} \Big) + O \Big( \frac{1}{n} \Big).
\end{align*}
Exponentiating this identity, and estimating the factor $\exp \big( O( \frac{1}{n}) \big)$ via the Maclaurin series $\sum_{k=0}^\infty \frac{x^k}{k!}$ for $\exp(x)$, we then have 
\begin{align}\label{further_asymp_pf_eqn_4}
\displaystyle \frac{\omega_{n-1}}{n \omega_n} = \frac{1}{\sqrt{2\pi n}} + O \Big( \frac{1}{n} \Big).
\end{align}
Combining \eqref{further_asymp_pf_eqn_2}, \eqref{further_asymp_pf_eqn_3}, and \eqref{further_asymp_pf_eqn_4} implies that
\begin{equation}\label{mfape_b}
\displaystyle \ln B_n = -\frac{1}{2} \ln n - \frac{1}{2} \ln(2\pi) + \frac{1}{\sqrt{2\pi n}} + O \Big( \frac{1}{n} \Big). 
\end{equation}
From \eqref{mfape_a} and \eqref{mfape_b},
\[
	\ln \Big( \frac{A_n}{B_n} \Big) = \ln(2\sqrt{e}) + \frac{\ln n}{2n} - \frac{1}{\sqrt{2\pi n}} + O \Big( \frac{1}{n}\Big),
\]
which gives the result. \end{proof}
\noindent As discussed above, Lemmas \ref{intersection piece} and \ref{further asymptotics} imply that $\rho_n^* > \rho_n$ for all sufficiently large $n$. For those $n$ with $\rho_n^* > \rho_n$, the definitions of $\mathcal{I}_n$ and $\mathcal{J}_n$ give $\rho_n^*$ as the unique root in the interval $(1,\infty)$ of the following degree-$n$ polynomial: 
\begin{equation}\begin{aligned}\label{definition of intersection point}
\rho_n^* (\rho_n^* - 1)^{n-1} &= \pi n(n-1) 2^{n-1} \Big( \frac{\omega_{n-1}}{\omega_{n-2}} \Big) \Big( \frac{K_{n+1}}{\omega_{n+1}} \Big)^{\frac{n}{n+1}},
\end{aligned}\end{equation}
where $K_{n+1}$ is as above, and where we have used the identity $\sigma_n = 2\pi \omega_{n-1}$ to derive the expression on the right-hand side of \eqref{definition of intersection point}. The constants in \eqref{definition of intersection point} have the following asymptotic behavior.

\begin{lemma}\label{asymptotics_lemma_1} 
Let $C_n$ be the $(n-1)$-st root of the right-hand side of \eqref{definition of intersection point}, that is,
\begin{align*}
\displaystyle C_n = 2[\pi n (n-1)]^{\frac{1}{n-1}} \Big( \frac{\omega_{n-1}}{\omega_{n-2}} \Big)^{\frac{1}{n-1}} \Big( \frac{K_{n+1}}{\omega_{n+1}} \Big)^{\frac{n}{(n+1)(n-1)}}. 	
\end{align*}
Then, $C_n = \sqrt{ \frac{2e}{\pi n} } \Big[ 1 + \frac{\ln n}{n} + O \Big( \frac{1}{n} \Big) \Big]$ as $n \to \infty$. In particular, $\lim\limits_{n \to \infty} C_n = 0$. 
\end{lemma}

\begin{proof} From the definition of $C_n$, 
\begin{align*}
\displaystyle \ln C_n &= \ln 2 + \frac{1}{n-1} \Big[ \ln \pi + \ln n + \ln(n-1) + \ln \omega_{n-1} - \ln \omega_{n-2} + \frac{n}{n+1} (\ln K_{n+1} - \ln \omega_{n+1}) \Big]. 
\end{align*}
By \eqref{logarithm shift}, \eqref{reciprocal shift}, \eqref{K minus ball}, and the cases $m=-1$ and $m=-2$ of Lemma \ref{ball asympototics}, we then have
\begin{align*}
\ln C_n &= \ln 2 + \frac{1}{2} \Big[ \frac{1}{n} + \frac{1}{n^2} + O \Big( \frac{1}{n^3} \Big) \Big] \Big[ - n \ln n + \ln \Big( \frac{e}{2\pi} \Big) n + 3 \ln n + \ln \Big( \frac{4 \pi^3}{3e} \Big) + O \Big( \frac{\ln n}{n} \Big) \Big]\\
&= - \frac{1}{2} \ln n + \frac{1}{2} \ln \Big( \frac{2e}{\pi} \Big) + \frac{\ln n}{n} + O \Big( \frac{1}{n} \Big),
\end{align*}
which implies the result. \end{proof}

\noindent Equation \eqref{definition of intersection point} can be solved in radicals to determine $\rho_n^*$ for $n \leq 4$. In general, let $C_n$ be as in Lemma \ref{asymptotics_lemma_1}, and, for each natural number $N$, let $f_N(\rho) = \rho^{1/N}(\rho - 1)$. The function $f_N$ is strictly increasing, and therefore invertible, on the interval $\big( \frac{1}{N+1},\infty \big)$, and $\rho_n^* = f_{n-1}^{-1}(C_n)$. We will use this identity to calculate $\rho_n^*$, for sufficiently large $n$, by an analytic method due to Lagrange and B\"{u}rmann.

\begin{theorem}[Lagrange--B\"{u}rmann \cite{Cauchy1932,LagrangeLegendre1799}]\label{lagrange_burmann}
Let $f$ be complex-analytic on a neighborhood of $a$ with $f'(a) \neq 0$. Then, for $c_k = \lim\limits_{z \to a} \frac{d^{k-1}}{dz^{k-1}} \Big|_{z = a} \Big[ \Big( \frac{z - a}{f(z) - f(a)} \Big)^k \Big]$, the inverse of $f$ in a neighborhood of $f(a)$ is given by $f^{-1}(w) = a + \sum_{k=1}^\infty c_k \frac{(w - f(a))^k}{k!}$. 
\end{theorem}
\noindent Sequences $\left(a_n\right)_{n=1}^{\infty}$ and $\left(b_n\right)_{n=1}^{\infty}$ are defined to be \textbf{asymptotically equivalent}, written $a_n \sim b_n$, if $\displaystyle \lim\limits_{n \to \infty} \frac{a_n}{b_n} = 1$. The Lagrange--B\"urmann formula applies to $f_{n-1}$ via the following result. 

\begin{lemma}\label{asymptotics_prop_3}
For each natural number $N$, let $f_N$ be the complex-analytic function in a neighborhood of $z = 1$ given by $f_N(z) = z^{1/N}(z - 1)$, with $z^{1/N}$ the principal branch of the $N$-th root, i.e., with $1^{1/N} = 1$. Then, $f_N$ is invertible near $z=1$, and, in a neighborhood of $f_N(1) = 0$, 
\begin{align*}
\displaystyle f_N^{-1}(y) = 1 + \sum_{k=1}^\infty \frac{d^{k-1}}{d\rho^{k-1}} \Big|_{\rho = 1} \Big[ \frac{1}{\rho^{k/N}} \Big] \frac{y^k}{k!} = 1 + y + \sum_{k=2}^\infty \frac{(-1)^{k-1} (\frac{k}{N})^{(k-1)}}{k!} y^k,
\end{align*}
where $x^{(m)}$ is the Pochhammer symbol for the rising factorial, i.e., $x^{(m)} = \prod_{j = 0}^{m-1} (x + j)$. The radius of convergence of this series is $\frac{N}{(N+1)^{1 + \frac{1}{N}}}$. 
\end{lemma}

\begin{proof} The invertibility of $f_N$ near $1$ follows from the fact that $f_N'(1) = 1$, and an elementary calculation gives the power series for $f_N^{-1}(y)$ via Theorem \ref{lagrange_burmann}. To determine its radius of convergence, we use the identity $(x)^{(m)} = \frac{\Gamma(x + m)}{\Gamma(x)}$ and compute the ratio of the magnitudes of successive coefficients:
\begin{equation}\label{asymptotics_prop_3_pf_eqn_1} 
\frac{\Big(\frac{\left(\frac{k+1}{N}\right)^{(k)}}{(k+1)!}\Big)}{\Big(\frac{\left(\frac{k}{N}\right)^{(k-1)}}{k!}\Big)} = \frac{1}{k+1} \frac{\Gamma(k + \frac{k}{N} + \frac{1}{N})}{\Gamma(k + \frac{k}{N} - 1)} \frac{\Gamma(\frac{k}{n})}{\Gamma(\frac{k}{N} + \frac{1}{N})}.
\end{equation}
A well-known approximation, which follows from Theorem \ref{binet_theorem}, is that $\frac{\Gamma(z + a)}{\Gamma(z + b)} \sim z^{a-b}$ as $z \to \infty$. As $k \to \infty$, the above ratio is therefore asymptotic to $\frac{1}{k} \Big( k + \frac{k}{N} \Big)^{1 + \frac{1}{N}} \Big( \frac{k}{N} \Big)^{-\frac{1}{N}} = \frac{(N+1)^{1 + \frac{1}{N}}}{N}$, which gives the radius of convergence. \end{proof}

\begin{lemma}\label{convergence_limit_lemma}
Let $\Phi_{N}(k) = \frac{Nk + N}{(N+1)k + 1 - N} \big( \frac{k}{k+1} \big)^{k-1}$. Then, $\Phi_{N}$ is nonincreasing for $k \geq 2$, and $\lim\limits_{k \to \infty} \Phi_{N}(k) = \frac{N}{(N+1)e}$.
\end{lemma}

\begin{proof} Let $\varphi_{N}(k) = \frac{Nk + N}{(N+1)k + 1 - N}$ and $\psi(k) = (k-1) \ln \big( \frac{k}{k+1} \big)$, so that $\Phi_{N}(k) = \varphi_{N}(k) e^{\psi(k)}$. An elementary calcuation gives $\varphi_{N}'(k) = \frac{-2N^2}{((N+1)k + 1 - N)^2} < 0$, which implies $\varphi_{N}$ is strictly decreasing in $k$, and $\psi'(k) = \ln \left( \frac{k}{k+1} \right) + \frac{k-1}{k(k+1)}$, which implies that $\psi'(2) = \ln \frac{2}{3} + \frac{1}{6} \approx -0.2388 < 0$ and $\lim\limits_{k \to \infty} \psi'(k) = 0$. We thus have $\psi''(k) = \frac{1}{k^2} + \frac{1}{k(k+1)} - \frac{2}{(k+1)^2} > 0$, which implies, together with the above, that $\psi'(k) < 0$ and, therefore, that $\psi$ is decreasing in $k$ for all $k \geq 2$. Together, these observations imply that $\Phi_{N}$ is decreasing in $k$ for $k \geq 2$. An elementary calculation with l'H\^opital's rule gives the limit. \end{proof}

\begin{lemma}\label{rho estimate lemma}
For $\rho_n^*$ as above, $\rho_n^* = 1 + \sqrt{\frac{2e}{\pi n}} + \sqrt{\frac{2e}{\pi}} \frac{\ln n}{n^{3/2}} + O \Big( \frac{1}{n^{3/2}} \Big)$ as $n \to \infty$.  
\end{lemma}

\begin{proof} As noted above, $\rho_n^* = f_{n-1}^{-1}(C_n)$, and Lemma \ref{asymptotics_lemma_1} implies that $C_n$ lies within the radius of convergence of the power series for $f_{n-1}^{-1}$ in Lemma \ref{asymptotics_prop_3} for sufficiently large $n$. Moreover, the ratio in \eqref{asymptotics_prop_3_pf_eqn_1} between the consecutive coefficients of the series satisfies 
\begin{equation*}
\displaystyle \frac{\Big(\frac{\left(\frac{k+1}{N}\right)^{(k)}}{(k+1)!}\Big)}{\Big(\frac{\left(\frac{k}{N}\right)^{(k-1)}}{k!}\Big)} =  \frac{(N+1)k + 1 - N}{nk + N} \prod\limits_{j=0}^{k-2} \frac{k + 1 + Nj}{k + Nj} \leq \frac{(N+1)k + 1 - N}{Nk + N} \left( \frac{k+1}{k} \right)^{k-1}. 
\end{equation*}
For $|y| < \frac{N}{(N+1)e}$, Lemma \ref{convergence_limit_lemma} therefore implies that $\Big|\frac{(\frac{k}{N})^{(k-1)}}{k!} y^k\Big|$ is decreasing in $k$ for $k \geq 2$. Together, these observations imply that, for sufficiently large $n$, 	
\begin{equation}\label{rho bound}
\displaystyle 1 + C_n - \frac{1}{n-1} C_n^2 \leq \rho_n^* \leq 1 + C_n.
\end{equation}
By Lemma \ref{asymptotics_lemma_1} and \eqref{reciprocal shift}, 
\begin{equation}\label{rho estimate lemma pf eqn 1}
\displaystyle \frac{1}{n-1} C_n^2 = \frac{2e}{\pi n} \Big[ \frac{1}{n} + \frac{1}{n^2} + O \Big( \frac{1}{n^3} \Big) \Big] \Big[ 1 + \frac{2 \ln n}{n} + O \Big( \frac{1}{n} \Big) \Big] = \frac{2e}{\pi n^2} + O \Big( \frac{\ln n}{n^3} \Big).
\end{equation}
Combining Lemma \ref{asymptotics_lemma_1}, \eqref{rho bound}, and \eqref{rho estimate lemma pf eqn 1} gives the result. \end{proof}


\begin{proposition}\label{h_n_prop}
Let $\mathcal{I}_n$, $\mathcal{J}_n$, and $\rho_n^*$ be as above, and let $h_n = \mathcal{I}_n(\rho_n^*) = \mathcal{J}_n(\rho_n^*)$. Then, $h_n = \frac{2 e^{e/\pi}}{\sqrt{3e}} \big( \frac{e}{n} \big)^n e^{-\sqrt{\frac{2en}{\pi}}} \big[ 1 - \sqrt{\frac{2e}{\pi n}} \ln n + O \big( \frac{1}{\sqrt{n}} \big) \big]$ as $n \to \infty$. In particular,
\begin{equation}\label{h_n_prop_eqn_2}
\displaystyle h_n \sim \frac{2 e^{e/\pi}}{\sqrt{3e}} \Big( \frac{e}{n} \Big)^n e^{-\sqrt{\frac{2en}{\pi}}} \sim \frac{e^{e/\pi}}{n!} \sqrt{\frac{8\pi n}{3e}} e^{-\sqrt{\frac{2en}{\pi}}}.
\end{equation}
\end{proposition}

\begin{proof} Because $h_n = \mathcal{J}_n(\rho_n^*) = 2\pi \omega_{n-1} \big( \frac{K_{n+1}}{\omega_{n+1}} \big)^{\frac{n}{n+1}} \frac{1}{(\rho_n^*)^n}$,
\begin{align}\label{h_n_prop_pf_eqn_1}
\displaystyle \ln h_n = \ln(2\pi) + \ln \omega_{n-1} + \frac{n}{n+1} (\ln K_{n+1} - \ln \omega_{n+1}) - n\ln \rho_n^*.
\end{align}
Lemma \ref{rho estimate lemma} and the Maclaurin series $\sum_{k=1}^\infty \frac{(-1)^{k-1}}{k} x^k$ for $\ln(1 + x)$ imply that 
\[
\ln \rho_n^* = \sqrt{\frac{2e}{\pi n}} - \frac{e}{\pi n} + \sqrt{\frac{2e}{\pi}} \frac{\ln n}{n^{3/2}} + O \Big( \frac{1}{n^{3/2}} \Big).
\]
Combining this with the asymptotic formulas for $\ln \omega_{n-1}$ and $\frac{n}{n+1} (\ln K_{n+1} - \ln \omega_{n+1})$ in Lemma \ref{ball asympototics} and \eqref{K minus ball} yields
\[
\ln h_n = - n\ln n + n - \sqrt{\frac{2en}{\pi}} + \frac{1}{2} \ln \Big( \frac{4}{3e} \Big) + \frac{e}{\pi} - \sqrt{\frac{2e}{\pi}} \frac{\ln n}{\sqrt{n}} + O \Big( \frac{1}{\sqrt{n}} \Big),
\]
which gives the asymptotic formula for $h_n$ above and implies the first equivalence in \eqref{h_n_prop_eqn_2}. Theorem \ref{binet_theorem} then implies the second equivalence in \eqref{h_n_prop_eqn_2}. \end{proof}

\noindent The best constant that may be valid in Theorem \ref{main theorem} is $\frac{\sigma_n}{\pi^n}$, corresponding to the antipodal map of the round sphere.   

\begin{proposition}\label{optimality_estimate}
As $n \to \infty$, the estimate in Proposition \ref{h_n_prop} for the constant in Theorem \ref{main theorem} is suboptimal by at most a factor of $\displaystyle \frac{2e^{\frac{e}{\pi}}}{\sqrt{6e}} e^{-\sqrt{\frac{2en}{\pi}}} \left( \sqrt{\frac{\pi}{2}} \right)^n \left( 2\pi n \right)^{\frac{1}{4}} \frac{1}{\sqrt{n !}}\left[1 - \sqrt{\frac{2e}{\pi}} \frac{\ln n}{\sqrt{n}} + O \Big( \frac{1}{n} \Big) \right]$. 
\end{proposition}

\begin{proof} The identity $\frac{\sigma_n}{\pi^n} = \frac{2}{\Gamma(\frac{n+1}{2})\pi^{\frac{n-1}{2}}}$, the asymptotic formula $\sqrt{\left(\frac{n}{n-1}\right)^n} = \sqrt{e}\left[1 + O \big( \frac{1}{n} \big) \right]$, and Theorem \ref{binet_theorem} imply that 
\begin{align}\label{opt_est_pf_eqn_1}
\frac{\sigma_n}{\pi^n} = \sqrt{\frac{2}{e}} \left( \sqrt{\frac{2e}{\pi(n-1)}} \right)^n \left[1 + O \Big( \frac{1}{n} \Big) \right] = \sqrt{2} \left(\sqrt{\frac{2e}{\pi n}} \right)^n \left[1 + O \Big( \frac{1}{n} \Big) \right] .
\end{align}
Dividing the asymptotic formula for $h_n$ in Proposition \ref{h_n_prop} by \eqref{opt_est_pf_eqn_1} gives the result. \end{proof}

\begin{remark}\label{bezdek remark}
As mentioned in Remark \ref{bezdek bound},  in dimensions $d \geq 3$, Bezdek \cite[Theorem 6.2]{Bezdek2013} improves the constant in Theorem \ref{pal-firey} from $K_d = \frac{2}{\sqrt{3} d!}$ to $F_d = \sqrt{ \frac{3\pi^{d-3} (d+2)!!}{(d+1)^2 (d!!)^2[(d-1)!!]^3} }$ for even $d$ and $F_d = \sqrt{ \frac{3\pi^{d-3}(d+1)!!}{2^{d-2}(d!!)^5} }$ for odd $d$, where $d!!$ denotes the product of all integers in the interval $[1,d]$ with the same parity as $d$. Letting $\overline{\mathcal{J}}_n$ be the function that results from substituting $F_{n+1}$ for $K_{n+1}$ in the definition of $\mathcal{J}_n$, and letting $\overline{\rho}_n^*$ be the unique $\rho > 1$ with $\overline{\mathcal{J}}_n(\overline{\rho}_n^*) = \mathcal{I}_n(\overline{\rho}_n^*)$, we have $\overline{\rho}_n^* > \rho_n^*$, because $\overline{\mathcal{J}}_n(\rho) > \mathcal{J}_n(\rho)$ for all $\rho > 1$. In particular, $\overline{\rho}_n^* > \rho_n$ whenever $\rho_n^* > \rho_n$. As $d \to \infty$, Theorem \ref{binet_theorem} and the identities $(2m)!! = 2^m m!$ and $(2m-1)!! = \frac{(2m-1)!}{(2m-2)!!}$ imply that
\[
\ln(d!!) = \left\{ \begin{array}{cc} \frac{1}{2} d\ln d - \frac{1}{2} d + \frac{1}{2} \ln d + \frac{1}{2} \ln \pi + O \big( \frac{1}{d} \big) & \textrm{if  } d \textrm{ is even,}\\ 
\frac{1}{2} d\ln d - \frac{1}{2} d + \frac{1}{2} \ln d + \frac{1}{2} \ln 2 + O \big( \frac{1}{d} \big) & \textrm{if } d \textrm{ is odd,} \end{array} \right. 
\]
and, therefore, that
\[
\ln F_{n+1} = \left\{ \begin{array}{cc} -n \ln n + \frac{1}{2} \ln \big( \frac{\pi e^2}{2} \big) n - \frac{7}{4} \ln n + \frac{1}{4} \ln \big( \frac{9}{8 \pi^3} \big) + O \big( \frac{1}{n} \big) & \textrm{is } n \textrm{ is even,}\\ 
-n \ln n + \frac{1}{2} \ln(\pi e^2) n  - \frac{7}{4} \ln n + \frac{1}{4} \ln \big( \frac{9}{8 \pi^5} \big) + O \big( \frac{1}{n} \big) & \textrm{is } n \textrm{ is odd.} \end{array} \right.
\]
We have the following parallel to the asymptotic formula in \eqref{K minus ball}: 
\[
\frac{n}{n+1} \left( \ln F_{n+1} - \omega_{n+1} \right) = \left\{ \begin{array}{cc} - \frac{1}{2} n \ln n + \frac{1}{2} \ln \big( \frac{e}{4} \big) n - \frac{1}{4} \ln n + \frac{1}{4} \ln \big( \frac{9}{2 \pi^3 e^2} \big) + O \big( \frac{\ln n}{n} \big) & \textrm{if } n \textrm{ is even,}\\ 
- \frac{1}{2} n \ln n + \frac{1}{2} \ln \big( \frac{e}{2} \big) n - \frac{1}{4} \ln n + \frac{1}{4} \ln \big( \frac{9}{8 \pi^5 e^2} \big) + O \big( \frac{\ln n}{n} \big) & \textrm{if } n \textrm{ is odd.} \end{array} \right.
\]
Letting $\overline{C}_n$ be the constant given by substituting $F_{n+1}$ for $K_{n+1}$ in the definition of $C_n$ in Lemma \ref{asymptotics_lemma_1}, it follows that, as $n \to \infty$,
\[
\overline{C}_n = \left\{ \begin{array}{cc} \sqrt{\frac{e}{n}} \Big[ 1 + \frac{3}{4}\frac{\ln n}{n} + O \big( \frac{1}{n} \big) \Big] & \textrm{if } n \textrm{ is even,}\\ 
\sqrt{\frac{2e}{n}} \Big[ 1 + \frac{3}{4}\frac{\ln n}{n} + O \big( \frac{1}{n} \big) \Big]  & \textrm{if } n \textrm{ is odd.} \end{array} \right. 
\]
An analysis as in Lemma \ref{rho estimate lemma} then gives
\[
\overline{\rho}_n^* = \left\{ \begin{array}{cc} 1 + \sqrt{\frac{e}{n}} + \frac{3\sqrt{e}}{4} \frac{\ln n}{n^{3/2}} + O \big( \frac{1}{n^{3/2}} \big) & \textrm{if } n \textrm{ is even,}\\ 
1 + \sqrt{\frac{2e}{n}} + \frac{3\sqrt{2e}}{4} \frac{\ln n}{n^{3/2}} + O \big( \frac{1}{n^{3/2}} \big) & \textrm{if } n \textrm{ is odd.} \end{array} \right.
\]
Letting $\overline{h}_n$ be the common value $\overline{\mathcal{J}}_n(\overline{\rho}_n^*) = \mathcal{I}_n(\overline{\rho}_n^*)$, we then have
\[
\overline{h}_n = \left\{ \begin{array}{cc} e^{(e-1)/2} \big( \frac{9}{2\pi n} \big)^{1/4} \big( \sqrt{\frac{\pi}{2}} \big)^n \big( \frac{e}{n} \big)^n e^{-\sqrt{en}} \big[ 1 - \frac{3\sqrt{e}}{4} \frac{\ln n}{\sqrt{n}} + O \big( \frac{1}{\sqrt{n}} \big) \big] & \textrm{if } n \textrm{ is even,}\\ 
e^e \big( \frac{9}{8\pi^3 n} \big)^{1/4} (\sqrt{\pi})^n \big( \frac{e}{n} \big)^n e^{-\sqrt{2en}} \big[ 1 - \frac{3\sqrt{2e}}{4} \frac{\ln n}{\sqrt{n}} + O \big( \frac{1}{\sqrt{n}} \big) \big] & \textrm{if } n \textrm{ is odd.} \end{array} \right. 
\]
As a result,
\[
\frac{\overline{h}_n}{\left(\frac{\sigma_n}{\pi^n}\right)} \sim \left\{ \begin{array}{cc} e^{\frac{e-1}{2}} \left( \frac{9}{2} \right)^{\frac{1}{4}} \left( \frac{\pi}{2} \right)^n e^{-\sqrt{en}} \frac{1}{\sqrt{n !}}& \textrm{if } n \textrm{ is even,}\\
e^e \sqrt{\frac{3}{4\pi}} \big( \frac{\pi}{\sqrt{2}} \big)^n e^{-\sqrt{2en}} \frac{1}{\sqrt{n !}} & \textrm{if } n \textrm{ is odd.} \end{array} \right. 
\]
Asymptotically, the constants given by Bezdek's improvement of Theorem \ref{pal-firey} still decay at a rate approximately proportional to $\frac{1}{\sqrt{n !}}$ relative to those given by the antipodal map of the round sphere, but they are improved by a factor of order $n^{-\frac{1}{4}}\left( \sqrt{\frac{\pi}{2}} \right)^n e^{(\sqrt{\frac{2e}{\pi}} - \sqrt{e})\sqrt{n}}$ for even $n$ and $n^{-\frac{1}{4}}\left( \sqrt{\pi} \right)^n e^{(\sqrt{\frac{2e}{\pi}} - \sqrt{2e})\sqrt{n}}$ for odd $n$ relative to those in Proposition \ref{h_n_prop}. 
\end{remark}


\section{Intrinsic versus extrinsic displacement of fixed-point-free maps}
\label{int_ext_appendix}


In this appendix, we study the invariant defined for fixed-point-free maps of convex hypersurfaces in Propositions \ref{enclosed volume bound} and \ref{area bound proposition}. 

\begin{proposition}\label{appendix int ext prop}
Let $\alpha:M \to M$ be a fixed-point-free map of a closed, convex hypersurface $M^{n}$ in $\real^{n+1}$, let $\rho(\alpha) = \max\limits_{x \in M} \frac{d_{M}(x,\alpha(x))}{|\alpha(x) - x|}$, and suppose one of the following holds: \\
\indent \textbf{(a)} $M$ is smooth. \\ 
\indent \textbf{(b)} $n$ is even. \\ 
\indent \textbf{(c)} $\alpha$ is an involution. \\
Then, $\rho(\alpha) > 1$. 
\end{proposition}

\begin{proof} If $d_M(x,\alpha(x)) = |\alpha(x) - x|$, then $M$ contains the line segment from $x$ to $\alpha(x)$. If $M$ is smooth, this implies the unit vector $\frac{\alpha(x) - x}{|\alpha(x) - x|}$ is a critical value of the map taking each vector in the unit tangent bundle of $M$ to its parallel in the unit sphere. Sard's theorem implies that the set of such critical values has measure $0$. However, if $\rho(\alpha) = 1$, this would be the case for all $x \in M$, and the surjectivity of the chordal Gauss map in Lemma \ref{chordal gauss map} would imply that $S^{n}$ consists entirely of critical values, a contradiction. This proves case (a).
\smallskip 

Suppose $n$ is even. If $M$ contained the line segment from $x$ to $\alpha(x)$ for each $x \in M$, which would follow from $\rho(\alpha) = 1$ as above, then translating from $\alpha(x)$ to $x$ along those segments would produce a homotopy from $\alpha$ to the identity map. However, Lemma \ref{hirsch-milnor} implies that $\alpha$ has degree $-1$, a contradiction. Case (b) follows.
\smallskip 

Conner and Floyd proved that every fixed-point-free involution of a sphere coincides at some point with any involution that is conjugate to the antipodal map \cite[Theorem 4.0]{ConnerFloyd1962}. Consequently, if $\alpha$ is an involution and $p$ is a point in the interior of the region bounded by $M$, there is $x \in M$ such that $\alpha(x)$ agrees with the map taking each $y \in M$ to the other point at which $\ell_{yp}$ intersects $M$. It follows that $M$ does not contain the line segment from $x$ to $\alpha(x)$, proving case (c). \end{proof}

\noindent It is also immediate that $\rho(\alpha) > 1$ when $M$ is strictly convex. However, the infimum of $\rho(\alpha)$ over all fixed-point-free maps $\alpha$ from $M$ into itself may be equal to $1$ even if $M$ is smooth and strictly convex: a small rotation of the Hopf fibres of any odd-dimensional sphere gives a fixed-point-free map for which the ratio $\frac{d_{M}(x,\alpha(x))}{|\alpha(x) - x|}$ is uniformly small. More broadly, flowing for a short time along a smooth, nowhere-vanishing vector field gives a map with this property on any odd-dimensional smooth convex hypersurface. In the remaining cases, i.e., non-involutive maps of odd-dimensional convex hypersurfaces of low regularity, this may present an obstacle to proving that $\rho(\alpha) > 1$ by approximation from the cases covered above.\\


\noindent \textbf{Statements and declarations.} The authors declare no competing interests or external sources of funding.

\bibliography{bibliography}

\providecommand{\bysame}{\leavevmode\hbox to3em{\hrulefill}\thinspace}
\providecommand{\MR}{\relax\ifhmode\unskip\space\fi MR }
\providecommand{\MRhref}[2]{%
  \href{http://www.ams.org/mathscinet-getitem?mr=#1}{#2}
}
\providecommand{\href}[2]{#2}
\begin{thebibliography}{10}

\bibitem{AignerZiegler1999}
Martin Aigner and G\"{u}nter~M. Ziegler, \emph{Proofs from \textnormal{The
  Book}}, Springer--Verlag, Berlin, 1999, viii+199 pp. Corrected reprint of the
  1998 original.

\bibitem{AkbulutKirby1979}
Selman Akbulut and Robion Kirby, \emph{An exotic involution of ${S}^4$},
  Topology \textbf{18} (1979), no.~1, 75--81.

\bibitem{Alexandrov1955}
A.~D. Alexandrov, \emph{Die innere {G}eometrie der konvexen {F}l\"{a}chen},
  Akademie--Verlag, Berlin, 1955 (German), xvii+522 pp.

\bibitem{Alzer2008}
Horst Alzer, \emph{Inequalities for the volume of the unit ball in
  $\mathbb{R}^n$. {I}{I}}, Mediterr. J. Math. \textbf{5} (2008), no.~4,
  395--413.

\bibitem{ArmanBondarenkoNazarovPrymakRadchenko2025}
Andrii Arman, Andriy Bondarenko, Fedor Nazarov, Andriy Prymak, and Danylo
  Radchenko, \emph{Small volume bodies of constant width}, Int. Math. Res. Not.
  IMRN (2025), no.~4, Paper No. rnaf020. 7 pp.

\bibitem{Berger1977}
Marcel Berger, \emph{Volume et rayon d'injectivit\'{e} dans les
  vari\'{e}t\'{e}s riemanniennes de dimension $3$}, Osaka Math. J. \textbf{14}
  (1977), no.~1, 191--200 (French).

\bibitem{Berger1980}
\bysame, \emph{Une borne inf\'{e}rieure pour le volume d'une vari\'{e}t\'{e}
  riemannienne en fonction du rayon d'injectivit\'{e}}, Ann. Inst. Fourier
  (Grenoble) \textbf{30} (1980), no.~3, 259--265 (French).

\bibitem{Besse1978}
Arthur~L. Besse, \emph{Manifolds all of whose geodesics are closed}, Ergebnisse
  der Mathematik und ihrer Grenzgebiete, vol.~93, Springer-Verlag, Berlin-New
  York, 1978, ix+262 pp. With appendices by D. B. A. Epstein, J.-P.
  Bourguignon, L. B\'{e}rard-Bergery, M. Berger, and J. L. Kazdan.

\bibitem{Bezdek2013}
K\'{a}roly Bezdek, \emph{{T}arski's plank problem revisited},
  Geometry---intuitive, discrete, and convex, Bolyai Society Mathematical
  Studies, vol.~24, J\'{a}nos Bolyai Mathematical Society, Budapest, 2013,
  pp.~45--64.

\bibitem{Binet1839}
Jacques Philippe~Marie Binet, \emph{M\'{e}moire sur les int\'{e}grales
  d\'{e}finies {E}ul\'{e}riennes et sur leur application \`{a} la th\'{e}orie
  des suites; ainsi qu'\`{a} l'\'{e}valuation des fonctions des grands
  nombres}, J. \'{E}c. Polytech. \textbf{16} (1839), no.~27, 123--343 (French).

\bibitem{Borisenko2025}
Alexander~A. Borisenko, \emph{Rigidity of closed convex hypersurfaces in
  multidimensional spaces of constant curvature}, J. Math. Phys. Anal. Geom.
  \textbf{21} (2025), no.~3, 267--275.

\bibitem{Brouwer1919}
L.~E.~J. Brouwer, \emph{\"{U}ber die periodischen {T}ransformationen der
  {K}ugel}, Math. Ann. \textbf{80} (1919), no.~1, 39--41 (German).

\bibitem{BuragoZalgaller1988}
Yu.~D. Burago and V.~A. Zalgaller, \emph{Geometric inequalities}, Grundlehren
  der mathematischen {W}issenschaften, vol. 285, Springer-Verlag, Berlin, 1988,
  xiv+331 pp. Translated from the Russian by A. B. Sosinskiĭ. Springer Series
  in Soviet Mathematics.

\bibitem{CappellShaneson1976}
Sylvain~E. Cappell and Julius~L. Shaneson, \emph{Some new four-manifolds}, Ann.
  of Math. \textbf{104} (1976), no.~1, 61--72.

\bibitem{Cauchy1813}
Augustin-Louis Cauchy, \emph{Sur les polygones et les poly\`{e}dres; {S}econd
  {M}\'{e}moire}, J. \'{E}c. Polytech. \textbf{9} (1813), no.~16, 87--98
  (French).

\bibitem{Cauchy1932}
\bysame, \emph{M\'{e}moire sur les fonctions de variables imaginaires et sur
  leurs d\'{e}veloppements en s\'{e}ries}, Oeuvres compl\`{e}tes d'{A}ugustin
  {C}auchy, 2e s\'{e}rie, vol.~13, Gauthier-Villars, Paris, 1932, pp.~176--203
  (French).

\bibitem{Chakerian1967}
G.~D. Chakerian, \emph{Inequalities for the difference body of a convex body},
  Proc. Amer. Math. Soc. \textbf{18} (1967), 879--884.

\bibitem{CohnVossen1936}
Stefan Cohn-Vossen, \emph{Bending of surfaces in the large}, Uspekhi Mat. Nauk
  \textbf{1} (1936), 33--76 (Russian).

\bibitem{ConnerFloyd1962}
P.~E. Conner and E.~E. Floyd, \emph{Fixed point free involutions and
  equivariant maps. {I}{I}}, Trans. Amer. Math. Soc. \textbf{105} (1962),
  222--228.

\bibitem{Crofton1868}
Morgan~W. Crofton, \emph{On the theory of local probability, applied to
  straight lines drawn at random in a plane; the methods used being also
  extended to the proof of certain new theorems in the integral calculus},
  Philos. Trans. Roy. Soc. London \textbf{158} (1868), 181--199.

\bibitem{Croke1987}
Christopher~B. Croke, \emph{Lower bounds on the energy of maps}, Duke Math. J.
  \textbf{55} (1987), no.~4, 901--908.

\bibitem{Croke1988}
\bysame, \emph{Area and the length of the shortest closed geodesic}, J.
  Differential Geom. \textbf{27} (1988), no.~1, 1--21.

\bibitem{Croke2002}
\bysame, \emph{The volume and lengths on a three sphere}, Comm. Anal. Geom.
  \textbf{10} (2002), no.~3, 467--474.

\bibitem{Croke2008}
\bysame, \emph{Small volume on big $n$-spheres}, Proc. Amer. Math. Soc.
  \textbf{136} (2008), no.~2, 715--717.

\bibitem{CrokeKatz2003}
Christopher~B. Croke and Mikhail Katz, \emph{Universal volume bounds in
  {R}iemannian manifolds}, Surveys in differential geometry, vol. VIII, Int.
  Press, Somerville, MA, 2003, pp.~109--137.

\bibitem{DibbleHoisington2026b}
James Dibble and Joseph~Ansel Hoisington, \emph{Mean curvature and closed
  geodesics in convex hypersurfaces}, arXiv: 2606.04112, submitted.

\bibitem{Firey1965}
William~J. Firey, \emph{Lower bounds for volumes of convex bodies}, Arch. Math.
  \textbf{16} (1965), 69--74.

\bibitem{GaoYau1986}
L.~Zhiyong Gao and Shing-Tung Yau, \emph{The existence of negatively {R}icci
  curved metrics on three manifolds}, Invent. Math. \textbf{85} (1986), no.~3,
  637--652.

\bibitem{Ghomi2017}
Mohammad Ghomi, \emph{Open problems in geometry of curves and surfaces},
  (2019), ghomi.math.gatech.edu/Papers/op.pdf.

\bibitem{Gray2004}
Alfred Gray, \emph{Tubes}, second ed., Birkh\"{a}user Verlag, Basil, 2004,
  vix+280 pp. With a preface by Vicente Miquel.

\bibitem{Gromov1983}
Mikhael Gromov, \emph{Filling {R}iemannian manifolds}, J. Differential Geom.
  \textbf{18} (1983), no.~1, 1--147.

\bibitem{GrovePetersen1991}
Karsten Grove and Peter Petersen, \emph{Manifolds near the boundary of
  existence}, J. Differential Geom. \textbf{33} (1991), no.~2, 379--394.

\bibitem{HardyLittlewoodPolya1952}
G.~H. Hardy, J.~E. Littlewood, and G.~P\'{o}lya, \emph{Inequalities}, second
  ed., Cambridge University Press, Cambridge, 1952, xii+324 pp.

\bibitem{Helgason1999}
Sigurdur Helgason, \emph{The {R}adon transform}, second ed., Progr. Math.,
  vol.~5, Birkh\"{a}user Boston, Inc., Boston, MA, 1999, xiv+188 pp.

\bibitem{Herglotz1943}
G.~Herglotz, \emph{\"{U}ber die {S}tarrheit der {E}ifl\"{a}chen}, Abh. Math.
  Sem. Hansischen Univ. \textbf{15} (1943), 127--129 (German).

\bibitem{HirschMilnor1964}
Morris~W. Hirsch and John Milnor, \emph{Some curious involutions of spheres},
  Bull. Amer. Math. Soc. \textbf{70} (1964), 372--377.

\bibitem{Howard2006}
Ralph Howard, \emph{Convex bodies of constant width and constant brightness},
  Adv. Math. \textbf{204} (2006), no.~1, 241--261.

\bibitem{Kerekjarto1919}
B.~von Ker\'{e}kj\'{a}rt\'{o}, \emph{\"{U}ber die periodischen
  {T}ransformationen der {K}reisscheibe und der {K}ugelfl\"{a}che}, Math. Ann.
  \textbf{80} (1919), no.~1, 36--38 (German).

\bibitem{KlainRota1997}
Daniel~A. Klain and Gian-Carlo Rota, \emph{A continuous analogue of {S}perner's
  theorem}, Comm. Pure Appl. Math. \textbf{50} (1997), no.~3, 205--223.

\bibitem{LagrangeLegendre1799}
Joseph-Louis Lagrange and Adrien-Marie Legendre, \emph{Rapport sur deux
  m\'{e}moires d'analyse du professeur {B}urmann}, M\'{e}m. Inst. Natl. Sci.
  Arts \textbf{2} (1799), 13--17 (French).

\bibitem{Livesay1960}
G.~R. Livesay, \emph{Fixed point free involutions on the $3$-sphere}, Ann. of
  Math. \textbf{72} (1960), 603--611.

\bibitem{Lohkamp1994}
Joachim Lohkamp, \emph{Metrics of negative {R}icci curvature}, Ann. of Math.
  \textbf{140} (1994), no.~3, 655--683.

\bibitem{LucardesiZucco2025}
Ilaria Lucardesi and Davide Zucco, \emph{Three quantitative versions of the
  {P}\'al inequality}, The Journal of Geometric Analysis \textbf{35} (2025),
  no.~3.

\bibitem{MartinezMaure1997}
Yves Martinez-Maure, \emph{Sur les h\'{e}rissons projectifs (enveloppes
  param\'{e}tr\'{e}es par leur application de {G}auss)}, Bull. Sci. Math.
  \textbf{121} (1997), no.~8, 585--601 (French).

\bibitem{Minkowski1903}
Hermann Minkowski, \emph{Volumen und {O}berfl\"ache}, Math. Ann. \textbf{57}
  (1903), no.~4, 447--495.

\bibitem{Minkowski1904}
\bysame, \emph{{\"U}ber die {K}\"orper konstanter {B}reite}, Mat. Sb.
  \textbf{25} (1904), 505--508.

\bibitem{Osserman1979}
Robert Osserman, \emph{Bonnesen-style isoperimetric inequalities}, Amer. Math.
  Monthly \textbf{86} (1979), no.~1, 1--29.

\bibitem{Paiva97}
Juan Carlos~\'Alvarez Paiva, \emph{Total mean curvature and closed geodesics},
  Bull. Belg. Math. Soc. Simon Stevin \textbf{4} (1997), no.~3, 373--377.

\bibitem{Pal1921}
Julius P{\'a}l, \emph{Ein {M}inimumproblem f\"{u}r {O}vale}, Math. Ann.
  \textbf{83} (1921), no.~3-4, 311--319 (German).

\bibitem{Pogorelov1973}
A.~V. Pogorelov, \emph{Extrinsic geometry of convex surfaces}, Translations of
  Mathematical Monographs, vol.~35, American Mathematical Society, Providence,
  1973, vi+669 pp.

\bibitem{Sacksteder1960}
Richard Sacksteder, \emph{On hypersurfaces with no negative sectional
  curvatures}, Amer. J. Math. \textbf{82} (1960), 609--630.

\bibitem{Santalo2004}
Luis~A. Santal\'{o}, \emph{Integral geometry and geometric probability}, second
  ed., Cambridge Math. Lib., Cambridge University Press, Cambridge, 2004,
  xx+404 pp. With a foreword by Mark Kac.

\bibitem{Schneider1993}
Rolf Schneider, \emph{Convex bodies: the {B}runn-{M}inkowski theory},
  Encyclopedia of Mathematics and Its Applications, vol.~44, Cambridge
  University Press, Cambridge, 1993, xiv+490 pp.

\bibitem{Senkin1972}
E.~P. Sen'kin, \emph{Rigidity of convex hypersurfaces}, Ukrain. Geometr. Sb.
  \textbf{12} (1972), 131--152 (Russian).

\bibitem{Treibergs1985}
Andrejs Treibergs, \emph{Estimates of volume by the length of shortest closed
  geodesics on a convex hypersurface}, Invent. Math. \textbf{80} (1985), no.~3,
  481--488.

\bibitem{Yang1954}
Chung-Tao Yang, \emph{On theorems of {B}orsuk--{U}lam,
  {K}akutani--{Y}amabe--{Y}ujob\^{o} and {D}yson. {I}}, Ann. of Math.
  \textbf{60} (1954), 262--282.

\end{thebibliography}
\bibliographystyle{amsplain}

\end{document}